\documentclass[a4paper,10pt]{amsart}

\usepackage{enumerate}
\usepackage{amsmath}
\usepackage{amscd}
\usepackage{amssymb}
\usepackage{latexsym}
\usepackage{stmaryrd} 
\usepackage{mathbbol}

\usepackage[arrow,curve,matrix,tips,2cell]{xy}
  \SelectTips{eu}{10} \UseTips
  \UseAllTwocells

\newtheorem{theorem}{Theorem}[section]
\newtheorem*{theorem*}{Theorem}
\newtheorem{lemma}[theorem]{Lemma}
\newtheorem*{lemma*}{Lemma}
\newtheorem{corollary}[theorem]{Corollary}
\newtheorem{proposition}[theorem]{Proposition}

\newtheorem{remark}[theorem]{Remark}
\newtheorem{definition}[theorem]{Definition}

\makeatletter
\def\revddots{\mathinner{\mkern1mu\raise\p@
\vbox{\kern7\p@\hbox{.}}\mkern2mu
\raise4\p@\hbox{.}\mkern2mu\raise7\p@\hbox{.}\mkern1mu}}
\makeatother 
\newcommand{\bgl}{\begin{equation}} 
\newcommand{\egl}{\end{equation}}
\newcommand{\bgloz}{\begin{equation*}} 
\newcommand{\egloz}{\end{equation*}}
\newcommand{\bgln}{\begin{eqnarray}} 
\newcommand{\egln}{\end{eqnarray}}
\newcommand{\bglnoz}{\begin{eqnarray*}} 
\newcommand{\eglnoz}{\end{eqnarray*}}
\newcommand{\btheo}{\begin{theorem}}
\newcommand{\etheo}{\end{theorem}}
\newcommand{\btheooz}{\begin{theorem*}}
\newcommand{\etheooz}{\end{theorem*}}
\newcommand{\blemma}{\begin{lemma}}
\newcommand{\elemma}{\end{lemma}}
\newcommand{\blemmaoz}{\begin{lemma*}}
\newcommand{\elemmaoz}{\end{lemma*}}
\newcommand{\bproof}{\begin{proof}}
\newcommand{\eproof}{\end{proof}}
\newcommand{\bbew}{\begin{beweis}}
\newcommand{\ebew}{\end{beweis}}
\newcommand{\bremark}{\begin{remark}\em}
\newcommand{\eremark}{\end{remark}}
\newcommand{\bdefin}{\begin{definition}}
\newcommand{\edefin}{\end{definition}}
\newcommand{\bprop}{\begin{proposition}}
\newcommand{\eprop}{\end{proposition}}
\newcommand{\bcor}{\begin{corollary}}
\newcommand{\ecor}{\end{corollary}}
\newcommand{\bfa}{\begin{cases}} 
\newcommand{\efa}{\end{cases}}
\newcommand{\cH}{\mathcal H}

\newcommand{\cL}{\mathcal L}
\newcommand{\cM}{\mathcal M}

\newcommand{\cR}{\mathcal R}

\def\Az{\mathbb{A}}
\def\Cz{\mathbb{C}}

\def\Nz{\mathbb{N}}

\def\Qz{\mathbb{Q}}
\def\Rz{\mathbb{R}}

\def\Zz{\mathbb{Z}}

\def\1z{\mathbb{1}}
\newcommand{\fA}{\mathfrak A}

\newcommand{\an}[1]{``#1''} 
\newcommand{\ti}{\tilde}

\newcommand{\ri}{\rightarrow}

\newcommand{\ma}{\mapsto} 
\newcommand\onto{\twoheadrightarrow} 
\newcommand\into{\hookrightarrow} 
\newcommand{\LRarr}{\Leftrightarrow} 

\def\SEMI{\mbox{$\times\kern-2pt\vrule height5pt width.6pt \kern3pt $}}

\newcommand{\id}{{\rm id}}

\newcommand{\tei}{\mid} 

\newcommand{\Ad}{{\rm Ad\,}}

\newcommand{\rk}{{\rm rk}\,}
\newcommand{\img}{{\rm im\,}}
\renewcommand{\ker}{{\rm ker}\,}

\newcommand{\reg}{^\times} 
\newcommand{\pos}{_{>0}} 
\newcommand{\nneg}{_{\geq 0}} 
\newcommand{\co}{^{\complement}} 
\newcommand{\abs}[1]{\lvert#1\rvert} 
\newcommand{\defeq}{\mathrel{:=}} 

\newcommand{\dop}{\text{: }} 
\newcommand{\falls}{\text{ if }} 
\newcommand{\sonst}{\text{ else}} 
\newcommand{\fa}{\text{ for all }} 
\newcommand{\ilim}{\varinjlim} 
\newcommand{\dotcup}{\ensuremath{\mathaccent\cdot\cup}} 
\newcommand{\extalg}{\Lambda^* \,} 
\newcommand{\Rbar}{\overline{R}} 
\newcommand{\lge}{\left\{} 
\newcommand{\rge}{\right\}} 
\newcommand{\lru}{\left(} 
\newcommand{\rru}{\right)} 
\newcommand{\leck}{\left[} 
\newcommand{\reck}{\right]} 
\newcommand{\lsp}{\left\langle} 
\newcommand{\rsp}{\right\rangle} 
\newcommand{\rukl}[1]{\lru #1 \rru} 
\newcommand{\eckl}[1]{\leck #1 \reck} 
\newcommand{\gekl}[1]{\lge #1 \rge} 
\newcommand{\spkl}[1]{\lsp #1 \rsp} 
\newcommand{\menge}[2]{\gekl{ #1 \dop #2 }} 
\newcommand{\eub}[1]{\underline{E}#1}

\DeclareMathOperator{\asmb}{asmb}
\DeclareMathOperator{\ind}{ind}
\DeclareMathOperator{\pt}{\{\bullet\}}

\DeclareMathOperator{\res}{res}

\begin{document}

\title[K-theory for ring C*-algebras]{K-theory for ring C*-algebras -- the case of number fields with higher roots of unity}

\author{Xin Li and Wolfgang L{\"u}ck}

\subjclass[2010]{Primary 46L05, 46L80; Secondary 11R04}

\thanks{\scriptsize{Research supported by the Deutsche Forschungsgemeinschaft (SFB 878), the ERC through AdG 267079 and the Leibniz-award of the second named author.}}

\begin{abstract}
We compute K-theory for ring C*-algebras in the case of higher roots of unity and thereby completely determine the K-theory for ring C*-algebras attached to rings of integers in arbitrary number fields.
\end{abstract}

\maketitle

\setcounter{tocdepth}{3}


\setlength{\parindent}{0pt} \setlength{\parskip}{0.5cm}

\section{Introduction}

Recently, a new type of constructions was introduced in the theory of operator algebras, so-called ring C*-algebras. The construction goes as follows: Given a ring $R$, take the Hilbert space $\ell^2(R)$ where $R$ is viewed as a discrete set. Consider the C*-algebra generated by all addition and multiplication operators induced by ring elements. This is the reduced ring C*-algebra of $R$. It is denoted by $\fA[R]$. Such an algebra was first introduced and studied by J. Cuntz in \cite{Cun} in the special case $R = \Zz$. As a next step, J. Cuntz and the first named author considered the case of integral domains satisfying a certain finiteness condition in \cite{Cu-Li1}. Motivating examples for such rings are given by rings of integers from algebraic number theory. It turns out that the associated ring C*-algebras carry very interesting structures and admit surprising alternative descriptions (see \cite{Cu-Li1} and \cite{Cu-Li2}). Finally, the most general case of rings without left zero-divisors was treated in \cite{Li}.

Of course, whenever new constructions of C*-algebras appear, one of the first problems is to compute their topological K-theory. Usually, this helps a lot in understanding the inner structure of the C*-algebras. In our situation, it even turns out that the ring C*-algebras attached to rings of integers are Kirchberg algebras satisfying the UCT (see \cite{Cu-Li1}, \S~3 and \cite{Li}, \S~5). For such C*-algebras, topological K-theory is a complete invariant. This is why computing K-theory is of particular interest and importance. The first K-theoretic computations were carried out in \cite{Cu-Li2} and \cite{Cu-Li3} for ring C*-algebras attached to rings of integers, but only in the special case where the roots of unity in the number field are given by $+1$ and $-1$. The reason why the general case could not be treated was that a K-theoretic computation for a certain group C*-algebra was missing.

In the present paper, we treat the remaining case of higher roots of unity. The missing ingredient is provided by \cite{La-Lü}, where for each number field, the K-theory of the group C*-algebra attached to the semidirect product of the additive group of the ring of integers by the multiplicative group of roots of unity in the number field has been computed. This computation serves as a starting point for our present paper and allows us to follow the strategy from \cite{Cu-Li2} to completely determine K-theory for ring C*-algebras associated with rings of integers in number fields.

Let us now formulate our main results. Let $K$ be a number field, i.e., a finite field extension of $\Qz$. The ring of integers in $K$, i.e., the integral closure of $\Zz$ in $K$, is denoted by $R$. Let $\fA[R]$ be the ring C*-algebra of $R$ defined at the beginning of the introduction (see also \S~\ref{review}). Moreover, the multiplicative group $K \reg$ always admits a decomposition of the form $K \reg = \mu \times \Gamma$ where $\mu$ is the group of roots of unity in $K$ and $\Gamma$ is a free abelian subgroup of $K \reg$. Here is our result treating the case of higher roots of unity:

\btheo
\label{thm1}
Assume that our number field contains higher roots of unity, i.e., $\abs{\mu}>2$. Then
$K_*(\fA[R]) \cong K_0(C^*(\mu)) \otimes_{\Zz} \extalg (\Gamma)$.
\etheo
The isomorphism above is meant as an isomorphism of $\Zz / 2 \Zz$-graded abelian groups. Here $K_*(\fA[R])$ is the $\Zz / 2 \Zz$-graded abelian group $K_0(\fA[R]) \oplus K_1(\fA[R])$. The exterior $\Zz$-algebra $\extalg (\Gamma)$ over $\Gamma$ is endowed with its canonical grading, the group $K_0(C^*(\mu))$ is trivially graded, and we take graded tensor products.

This theorem is the main result of this paper. In combination with the results from \cite{Cu-Li2} and \cite{Cu-Li3}, it gives the following complete description of the K-theory of $\fA[R]$ without restrictions on $\mu$:
\btheo
\label{thm2}
With the same notations as in the previous theorem (but without the assumption $\abs{\mu}>2$), we have
\bgloz
  K_*(\fA[R])
  \cong 
  \bfa 
    K_0(C^*(\mu)) \otimes_{\Zz} \extalg (\Gamma) & \falls \# \gekl{v_{\Rz}} \text{ is even}, \\
    \extalg (\Gamma) & \falls \# \gekl{v_{\Rz}} \text{ is odd},
  \efa
\egloz
again as $\Zz / 2 \Zz$-graded abelian groups.
\etheo
Here $\# \gekl{v_{\Rz}}$ denotes the number of real places of $K$. Note that since we are just identifying $\Zz / 2 \Zz$-graded abelian groups, our K-theoretic formulas could be further simplified. But we present the formulas in this way because this description of K-theory naturally comes out of our computations.

Using the classification result for UCT Kirchberg algebras due to E. Kirchberg and C. Phillips (see \cite[Chapter~8]{Ror}), we obtain from our K-theoretic computations
\bcor
\label{cor}
Given two arbitrary rings of integers $R_1$ and $R_2$ in number fields, the corresponding ring C*-algebras $\fA[R_1]$ and $\fA[R_2]$ are always isomorphic.
\ecor

This result should be contrasted with the observation due to J. Cuntz and C. Deninger that the C*-dynamical system $(\fA[R],\Rz,\sigma)$ (the higher dimensional analogue of the system introduced in \cite{Cun}, it is also the analogue for ring C*-algebras of the system introduced in \cite{C-D-L}) determines the number field in the case of Galois extensions of $\Qz$.

This paper is structured as follows: We start with a review on ring C*-algebras attached to rings of integers (\S~\ref{review}). In the main part of the paper (\S~\ref{mainsec}), we compute the K-theory for ring C*-algebras following the same strategy as in \cite{Cu-Li2}. Here we use results from \cite{La-Lü} on the K-theory of certain group C*-algebras (\S~\ref{K(group-C)}). In the last section, we show how the formalism of the Baum-Connes conjecture can be used to deduce injectivity for certain homomorphisms on the level of K-theory. To do so, we apply tools from algebraic topology. This injectivity result is used in \S~\ref{mainsec}, but its proof is given in \S~\ref{injectivity} because it is somewhat independent from the main part.

\section{Review}
\label{review}

From now on, let $K$ be a number field, i.e., a finite field extension of the rational numbers. Let $R$ be the ring of integers of $K$, i.e., $R$ is the integral closure of $\Zz$.

First of all, let us recall the construction of ring C*-algebras: We consider the Hilbert space $\ell^2(R)$ with its canonical orthonormal basis $\menge{\varepsilon_r}{r \in R}$. Then we use addition and multiplication in $R$ to define unitaries $U^b$ via $\varepsilon_r \ma \varepsilon_{b+r}$ and isometries $S_a$ via $\varepsilon_r \ma \varepsilon_{ar}$ for all $b$ in $R$, $a$ in $R\reg = R \setminus \gekl{0}$. The (reduced) ring C*-algebra is given by $\fA[R] \defeq C^*(\gekl{U^b},\gekl{S_a})$, the smallest involutive norm-closed algebra of bounded operators on $\ell^2(R)$ containing the families $\gekl{U^b}$ and $\gekl{S_a}$. Note that we write $\gekl{U^b}$ for $\menge{U^b}{b \in R}$ and $\gekl{S_a}$ for $\menge{S_a}{a \in R\reg}$. We use analogous notations for other families of generators as well. It turns out that $\fA[R]$ is isomorphic to the universal C*-algebra generated by unitaries $\menge{u^b}{b \in R}$ and isometries $\menge{s_a}{a \in R \reg}$ satisfying
\bglnoz
  &I.& u^b s_a u^d s_c = u^{b+ad} s_{ac} \\
  &II.& \sum u^b s_a s_a^* u^{-b} = 1
\eglnoz
where we sum over $R/aR = \menge{b+aR}{b \in R}$ in II. More precisely, Relation I implies that each of the summands $u^b s_a s_a^* u^{-b}$ only depends on the coset $b+aR$, not on the particular representative of the coset. Thus we obtain one summand for each coset in $R/aR$ and we sum up these elements in Relation II. Here $aR$ is the principal ideal of $R$ generated by $a$. We write $e_a$ for the range projection $s_a s_a^*$ of $s_a$.

From the description of $\fA[R]$ as a universal C*-algebra, it follows that $\fA[R] \cong C(\Rbar) \rtimes^e (R \rtimes R \reg)$. Here $\rtimes^e$ stands for \an{semigroup crossed product by endomorphisms}. The $R \rtimes R\reg$-action we consider is given by affine transformations as follows: $R$ sits as a subring in its profinite completion $\Rbar$ and thus acts additively and multiplicatively. Furthermore, it follows from this that $\fA[R]$ is Morita equivalent to the crossed product $C_0(\Az_f) \rtimes K \rtimes K \reg$, where $\Az_f$ is the finite adele space over $K$ and $K \rtimes K \reg$ act on $C_0(\Az_f)$ via affine transformations.

At this point, the duality theorem enters the game. It says that $C_0(\Az_{\infty}) \rtimes K \rtimes K \reg$ is Morita equivalent to $C_0(\Az_f) \rtimes K \rtimes K \reg$. For the first crossed product, we let $K \rtimes K \reg$ act on $C_0(\Az_{\infty})$ via affine transformations where $\Az_\infty$ is the infinite adele space of $K$. So on the whole, we obtain $\fA[R] \sim_M C_0(\Az_{\infty}) \rtimes K \rtimes K \reg$. Actually, the crossed products $C_0(\Az_{\infty}) \rtimes K$ and $C(\Rbar) \rtimes R$ are Morita equivalent in a $R\reg$-equivariant way. The reason is that the imprimitivity bimodule from \cite[\S~4]{Cu-Li2} carries a canonical $R\reg$-action which is compatible with the $R\reg$-actions on $C_0(\Az_{\infty}) \rtimes K$ and $C(\Rbar) \rtimes R$. As a consequence, we get that for every (multiplicative) subgroup $\Gamma$ of $K\reg$, the crossed products $C_0(\Az_{\infty}) \rtimes K \rtimes \Gamma$ and $C_0(\Gamma \cdot \Rbar) \rtimes (\Gamma \cdot R) \rtimes \Gamma$ are Morita equivalent (see \cite[Theorem~4.1]{Cu-Li2}). Here $\Gamma \cdot \Rbar$ is the subring of $\Az_f$ generated by $\Gamma$ and $\Rbar$, and $\Gamma \cdot R$ is the subring of $K$ generated by $\Gamma$ and $R$.

\section{K-theory for certain group C*-algebras}
\label{K(group-C)}

We now turn to the case of a number field $K$ and present the proof of Theorem~\ref{thm1}. Let $R$ be the ring of integers in $K$. Moreover, let $\mu$ be the group of roots of unity in $K$. This group is always a finite cyclic group generated by a root of unity, say $\zeta$.

The starting point for our K-theoretic computations is the work of M. Langer and the second named author on the K-theory of certain group C*-algebras. More precisely, in \cite{La-Lü}, the K-theory of the group C*-algebra of $R \rtimes \mu$ has been computed. Here $R \rtimes \mu$ is the semidirect product obtained from the multiplicative action of $\mu$ on the additive group $(R,+)$. The corresponding group C*-algebra is denoted by $C^*(R \rtimes \mu)$. It is very useful for our purposes that it is even possible to give an almost complete list of generators for the corresponding K-groups. Let us now summarize the results from \cite{La-Lü}.

To do so, we first need to introduce some notation. The additive group $(R,+)$ of our ring $R$ sits as a subgroup in $R \rtimes \mu$. Let $\iota: R \to R \rtimes \mu$ be the canonical inclusion, and denote by $\iota_*$ the homomorphism $K_0(C^*(R)) \to K_0(C^*(R \rtimes \mu))$ induced by $\iota$ on $K_0$. Moreover, given a finite subgroup $M$ of $R \rtimes \mu$, consider the canonical projection $M \onto \gekl{e}$ from $M$ onto the trivial group. This projection induces a homomorphism $C^*(M) \to \Cz$ of the group C*-algebras, hence a homomorphism on $K_0$: $K_0(C^*(M)) \to K_0(\Cz)$. Let us denote the kernel of this homomorphism by $\widetilde{R}_\Cz(M)$. The canonical inclusion $M \to R \rtimes \mu$ induces a homomorphism $\iota_M: C^*(M) \to C^*(R \rtimes \mu)$, hence a homomorphism $K_0(C^*(M)) \to K_0(C^*(R \rtimes \mu))$. Restricting this homomorphism to $\widetilde{R}_\Cz(M)$, we obtain $(\iota_M)_*: \widetilde{R}_\Cz(M) \to K_0(C^*(R \rtimes \mu))$. Here are the main results from \cite{La-Lü}:

\btheo[Langer-L\"uck]
\label{LL}
With the notations from above, we have
\begin{enumerate}
\item[($*$)]
$K_0(C^*(R \rtimes \mu))$ is finitely generated and torsion-free.
\item[($**$)]
Let $\cM$ be the set of conjugacy classes of maximal finite subgroups of $R \rtimes \mu$. Then $\sum_{(M) \in \cM} (\iota_M)_*: \bigoplus_{(M) \in \cM} \widetilde{R}_\Cz(M) \to K_0(C^*(R \rtimes \mu))$ is injective, i.e., for every $(M) \in \cM$, the map $(\iota_M)_*$ is injective and for every $(M_1)$, $(M_2)$ in $\cM$ with $(M_1) \neq (M_2)$ we have $\img((\iota_{M_1})_*) \cap \img((\iota_{M_2})_*) = \gekl{0}$. Moreover, $\img(\iota_*) \cap \rukl{\sum_{(M) \in \cM} \img((\iota_M)_*)} = \gekl{0}$ and $\img(\iota_*) + \rukl{\sum_{(M) \in \cM} \img((\iota_M)_*)}$ is of finite index in $K_0(C^*(R \rtimes \mu))$.
\item[($*$$*$$*$)]
$K_1(C^*(R \rtimes \mu))$ vanishes.
\end{enumerate}
\etheo
\bproof
($*$) is \cite[Theorem~0.1, (iii)]{La-Lü}. ($**$) is \cite[Theorem~0.1, (ii)]{La-Lü}. Note that the maps $\iota_M$ in our notation are denoted by $i_M$ in \cite{La-Lü}, and that $\iota$ in our notation is denoted by $k$ in \cite{La-Lü}. The group $\widetilde{R}_\Cz(M)$ coincides with the corresponding one in \cite{La-Lü} upon the canonical identification of the representation ring $R_\Cz(M)$ of $M$ with $K_0(C^*(M))$ as abelian groups. Furthermore, ($*$$*$$*$) is  \cite[Theorem~0.1, (iv)]{La-Lü}.
\eproof

Let us now describe $K_0(C^*(R \rtimes \mu))$ in a way which is most convenient for our K-theoretic computations. The idea is to use ($*$) and ($**$) from Theorem~\ref{LL} to decompose $K_0(C^*(R \rtimes \mu))$ into direct summands. However, we cannot simply use the subgroups $\img(\iota_*)$ and $\img((\iota_M)_*)$ for $(M) \in \cM$ which appear in ($**$) because these subgroups might not be direct summands. To solve this problem, we proceed as follows: First of all, we set
\bgl
\label{defKinf}
  K_{inf} \defeq \menge{x \in K_0(C^*(R \rtimes \mu))}{\exists \ N \in \Zz \pos \text{ such that } Nx \in \img(\iota_*)}.
\egl

Now take a finite subgroup $M$ of $R \rtimes \mu$. It has to be a cyclic group. Let $(b,\zeta^i)$ in $R \rtimes \mu$ be a generator of $M$. Note that $i=m/\abs{M}$ (up to multiples of $m$), where $m = \abs{\mu}$. Let $\chi$ be a character of $\Zz/\abs{M}\Zz$, and denote by $p_{\chi}(u^b s_{\zeta^i})$ the spectral projection $\tfrac{1}{\abs{M}} \sum_{j=0}^{\abs{M}-1} \chi(j+\abs{M}\Zz) (u^b s_{\zeta^i})^j$. Then $\img((\iota_M)_*)$ is generated by $\menge{\eckl{p_{\chi}(u^b s_{\zeta^i})}}{1 \neq \chi \in \widehat{\Zz/\abs{M}\Zz}}$. Here $\eckl{\cdot}$ denotes the $K_0$-class of the projection in question and $1 \in \widehat{\Zz/\abs{M}\Zz}$ is the trivial character.

It is then clear that the $K_0$-classes $\gekl{\eckl{p_{\chi}(u^b s_{\zeta^i})}}$ for $(\mu) \neq (M) \in \cM$, $M = \spkl{(b,\zeta^i)}$ and $1 \neq \chi \in \widehat{\Zz/\abs{M}\Zz}$ form a $\Zz$-basis of $\sum_{(\mu) \neq (M) \in \cM} \img((\iota_M)_*)$. Let us enumerate these $K_0$-classes $\gekl{\eckl{p_{\chi}(u^b s_{\zeta^i})}}$ by $y_1, y_2, y_3, \dotsc, y_{\rk_{fin}^{\co}}$, where $\rk_{fin}^{\co}$ is the rank of $\sum_{(\mu) \neq (M) \in \cM} \img((\iota_M)_*)$. As $K_0(C^*(R \rtimes \mu))$ is free abelian, we can recursively find $\overline{y}_1, \overline{y}_2, \overline{y}_3, \dotsc, \overline{y}_{\rk_{fin}^{\co}}$ in $K_0(C^*(R \rtimes \mu))$ such that for every $1 \leq j \leq \rk_{fin}^{\co}$,
\bglnoz
  && K_{inf} + \spkl{\overline{y}_1, \dotsc, \overline{y}_j} \\
  &=& \menge{x \in K_0(C^*(R \rtimes \mu))}{\exists \ N \in \Zz \pos \text{ such that } Nx \in K_{inf} + \spkl{y_1, \dotsc, y_j}}.
\eglnoz
By construction, these elements $\overline{y}_1, \overline{y}_2, \overline{y}_3, \dotsc, \overline{y}_{\rk_{fin}^{\co}}$ are linearly independent. We set $K_{fin}^{\co} = \spkl{\overline{y}_1, \overline{y}_2, \overline{y}_3, \dotsc, \overline{y}_{\rk_{fin}^{\co}}}$. By construction, $K_{inf} \cap K_{fin}^{\co} = \gekl{0}$. Finally, $\menge{\eckl{p_{\chi}(s_{\zeta})}}{1 \neq \chi \in \widehat{\Zz/m\Zz}}$ is a $\Zz$-basis of $\img((\iota_\mu)_*)$. Enumerate the elements $\eckl{p_{\chi}(s_{\zeta})}$, $1 \neq \chi \in \widehat{\Zz/m\Zz}$, by $z_1, \dotsc, z_{m-1}$. Again, there exist $\overline{z}_1, \dotsc, \overline{z}_{m-1}$ in $K_0(C^*(R \rtimes \mu))$ with the property that for every $1 \leq l \leq m-1$,
\bglnoz
  && K_{inf} + K_{fin}^{\co} + \spkl{\overline{z}_1, \dotsc, \overline{z}_l} \\
  &=& \menge{x \in K_0(C^*(R \rtimes \mu))}{\exists \ N \in \Zz \pos \text{ s.t. } Nx \in K_{inf} + K_{fin}^{\co} + \spkl{z_1, \dotsc, z_l}}.
\eglnoz
It is again clear that $\overline{z}_1, \dotsc, \overline{z}_{m-1}$ are linearly independent. We set $K_{fin}^\mu = \spkl{\overline{z}_1, \dotsc, \overline{z}_{m-1}}$. By construction, we have $(K_{inf} + K_{fin}^{\co}) \cap K_{fin}^\mu = \gekl{0}$. Thus $K_0(C^*(R \rtimes \mu)) = K_{inf} \oplus K_{fin}^{\co} \oplus K_{fin}^\mu$ as $\img(\iota_*) + \rukl{\sum_{(M) \in \cM} \img((\iota_M)_*)}$ is of finite index in $K_0(C^*(R \rtimes \mu))$ by ($**$) from Theorem~\ref{LL}.

\section{K-theory for ring C*-algebras}
\label{mainsec}

Let us recall the strategy of the previous K-theoretic computations from \cite{Cu-Li2} for number field without higher roots of unity. We will use the same strategy to treat the case of higher roots of unity.

The first step is to compute K-theory for the sub-C*-algebra $C^*(\gekl{u^b},s_{\zeta},\gekl{e_a})$ of $\fA[R]$. Recall that $e_a$ is the range projection of $s_a$, i.e., $e_a = s_a s_a^*$. This sub-C*-algebra can be identified with the inductive limit of the system given by the algebras $C^*(\gekl{u^b},s_{\zeta},e_a)$ for $a \in R \reg$. Moreover, we can prove that for fixed $a$ in $R \reg$, the algebra $C^*(\gekl{u^b},s_{\zeta},e_a)$ is isomorphic to a matrix algebra over $C^*(\gekl{u^b},s_{\zeta}) \cong C^*(R \rtimes \mu)$. In this situation, Theorem~\ref{LL} allows us to compute K-theory for $C^*(\gekl{u^b},s_{\zeta},\gekl{e_a})$ using its inductive limit structure.

The next step is to use the duality theorem (\S~\ref{review}) to pass over to the infinite adele space. The main point is to prove that the additive action of $K$ is negligible for K-theory, i.e., that the canonical homomorphism
\bgloz
  C_0(\Az_{\infty}) \rtimes K \reg \ri C_0(\Az_{\infty}) \rtimes K \rtimes K \reg
\egloz
induces an isomorphism on K-theory, at least rationally. This is good enough once we can show that all the K-groups are torsion-free. At this point, we need to know that the canonical homomorphism $C_0(\Az_{\infty}) \rtimes \mu \to C_0(\Az_{\infty}) \rtimes K \rtimes \mu$ is injective on K-theory. The proof of this statement is postponed to \S~\ref{injectivity}.

The last step is to compute K-theory for $C_0(\Az_{\infty}) \rtimes K \reg$ using homotopy arguments and the Pimsner-Voiculescu exact sequence. As we know that $\fA[R]$ is Morita equivalent to $C_0(\Az_{\infty}) \rtimes K \rtimes K \reg$, we finally obtain the K-theory for the ring C*-algebra $\fA[R]$.

\subsection{Identifying inductive limits}

The first step is to compute the K-theory of $C^*(\gekl{u^b},s_{\zeta},\gekl{e_a})$ using its inductive limit structure. Here $C^*(\gekl{u^b},s_{\zeta},\gekl{e_a})$ is the sub-C*-algebra of $\fA[R]$ generated by $\menge{u^b}{b \in R}$, $s_{\zeta}$ and $\menge{e_a}{a \in R \reg}$. First of all, note that for all $a$ and $c$ in $R \reg$, we have $e_a = \sum_{b+cR \in R/cR} u^{ab} e_{ac} u^{-ab}$. Just conjugate Relation II in \S~\ref{review} (for $c$ in place of $a$), $1=\sum u^b e_c u^{-b}$, by $s_a$. Therefore, the C*-algebras $C^*(\gekl{u^b},s_{\zeta},e_a)$ for $a$ in $R \reg$ and the inclusion maps
\bgloz
  \iota_{a,ac}: C^*(\gekl{u^b},s_{\zeta},e_a) \ri C^*(\gekl{u^b},s_{\zeta},e_{ac})
\egloz
form an inductive system. Here $R \reg$ is ordered by divisibility. It is clear that the inductive limit of this system can be identified with $C^*(\gekl{u^b},s_{\zeta},\gekl{e_a})$. Thus our goal is to compute the K-theory of $C^*(\gekl{u^b}, s_{\zeta}, e_a)$ and to determine the structure maps $\iota_{a,ac}$ on K-theory. Note that $C^*(\gekl{u^b}, s_{\zeta}, e_a)$ is obtained from $C^*(\gekl{u^b}, s_{\zeta})$ by adding one single projection $e_a$ and not the whole set of projections $\gekl{e_a}$. 

Let $a$ and $c$ be arbitrary elements in $R \reg$. Choose a minimal system $\cR_a$ of representatives for $R/aR$ in $R$. \an{Minimal} means that for arbitrary elements $b_1$ and $b_2$ in $\cR_a$, the difference $b_1-b_2$ lies in $aR$ (if and) only if $b_1=b_2$. We always choose $\cR_a$ in such a way that $0$ is in $\cR_a$. 

Using the decomposition $R = \dotcup_{b \in \cR_a} (b+aR)$ and the inverse of the isomorphism $\ell^2(R) \cong \ell^2(b+aR); \ \varepsilon_r \ma \varepsilon_{b+ar}$, we can construct the unitary 
\bgloz
  \ell^2(R) = \bigoplus_{b \in \cR_a} \ell^2(b+aR) \cong \bigoplus_{\cR_a} \ell^2(R).
\egloz
Conjugation with this unitary gives rise to an isomorphism 
\bgloz
  \cL(\ell^2(R)) \cong \cL(\ell^2(R/aR)) \otimes \cL(\ell^2(R)), T \ma \sum_{b,b' \in \cR_a} e_{b,b'} \otimes (s_a^* u^{-b} T u^{b'} s_a).
\egloz
Here $e_{b,b'}$ is the canonical rank $1$ operator in $\cL(\ell^2(R/aR))$ corresponding to $b+aR$ and $b'+aR$ sending a vector $\xi$ in $\ell^2(R/aR)$ to $\spkl{\xi,\varepsilon_{b'+aR}} \varepsilon_{b+aR}$. In this formula, $\menge{\varepsilon_{b+aR}}{b \in \cR_a}$ is the canonical orthonormal basis of $\ell^2(R/aR)$.

Let us denote the restriction of this isomorphism to $C^*(\gekl{u^b}, s_{\zeta}, e_{ac})$ by $\vartheta_{ac,c}$.

\blemma
\label{bd--d}
For every $a$ and $c$ in $R \reg$, the image of $\vartheta_{ac,c}$ is $\cL(\ell^2(R/aR)) \otimes (C^*(\gekl{u^b}, s_{\zeta}, e_c))$. Thus $\vartheta_{ac,c}$ induces an isomorphism
\bgloz
  C^*(\gekl{u^b}, s_{\zeta}, e_{ac}) \cong \cL(\ell^2(R/aR)) \otimes (C^*(\gekl{u^b}, s_{\zeta}, e_c)).
\egloz
\elemma

Since $R/aR$ is always finite, we know that $\cL(\ell^2(R/aR))$ is just a matrix algebra. So it does not matter which tensor product we choose. 

\bproof
A direct computation yields 
\bglnoz
  && \vartheta_{ac,c}(u^b e_a u^{-b'}) = e_{b,b'} \otimes 1 \fa b,b' \in \cR_a; \\
  && \vartheta_{ac,c}(u^{ab}) = 1 \otimes u^b \fa b \in R; \\
  && \vartheta_{ac,c}(\sum_{b \in \cR_a} u^b e_a u^{- \zeta b} s_{\zeta}) = 1 \otimes s_{\zeta}; \\
  && \vartheta_{ac,c}(\sum_{b \in \cR_a} u^b e_{ac} u^{-b}) = 1 \otimes e_c. 
\eglnoz
Our claim follows from the observation that $C^*(\gekl{u^b}, s_{\zeta}, e_{ac})$ is generated by 
\bgloz
  u^b e_a u^{-b'} \ (b,b' \in \cR_a) \text{; } u^{ab} \ (b \in R) \text{; } \sum_{b \in \cR_a} u^b e_a u^{- \zeta b} s_{\zeta} 
  \text{ and } \sum_{b \in \cR_a} u^b e_{ac} u^{-b}.
\egloz
\eproof

Let us now fix minimal systems of representatives $\cR_a$ for every $a$ in $R \reg$ as explained before the previous lemma (we will always choose $0 \in \cR_a$). As $R/aR$ is finite for every $a$ in $R \reg$, we know that $\cL(\ell^2(R/aR))$ is simply a matrix algebra of finite dimension. Thus we can use the previous lemma to identify $K_0(C^*(\gekl{u^b}, s_{\zeta}, e_a))$ and $K_0(C^*(\gekl{u^b}, s_{\zeta}))$ via $(\rho_{1,a})_*^{-1} (\vartheta_{a,1})_*$. Here $\rho_{c,a}$ (for $a$ and $c$ in $R \reg$) is the canonical homomorphism 
\bgloz
  C^*(\gekl{u^b}, s_{\zeta}, e_c) \ri \cL(\ell^2(R/aR)) \otimes (C^*(\gekl{u^b}, s_{\zeta}, e_c)); \ x \ma e_{0,0} \otimes x.
\egloz

\blemma
\label{a,ac->1,c}
We have 
\bgloz
  (\rho_{1,ac})_*^{-1} (\vartheta_{ac,1})_* (\iota_{a,ac})_* (\vartheta_{a,1}^{-1})_* (\rho_{1,a})_* = (\rho_{1,c})_*^{-1} (\vartheta_{c,1})_* (\iota_{1,c})_*.
\egloz
\elemma
In other words, under the K-theoretic identifications above, the map $(\iota_{a,ac})_*$ corresponds to $(\iota_{1,c})_*$. This observation is helpful because it says that we only have to determine the homomorphisms $\iota_{1,c}$ on K-theory. 
\bproof
It is immediate that 
\bgl
\label{axc=ac1}
  \rho_{c,a} = \vartheta_{ac,c} \circ \Ad(s_a)
\egl
as homomorphisms $C^*(\gekl{u^b}, s_{\zeta}, e_c) \ri \cL(\ell^2(R/aR)) \otimes (C^*(\gekl{u^b}, s_{\zeta}, e_c))$. Here we mean by $\Ad(s_a)$ the homomorphism $C^*(\gekl{u^b}, s_{\zeta}, e_c) \ri C^*(\gekl{u^b}, s_{\zeta}, e_{ac}); \ x \ma s_a x s_a^*$. It would be more precise to write $\Ad(s_a) \vert_{C^*(\gekl{u^b}, s_{\zeta}, e_c)}$, but it will become clear from the context on which domain $\Ad(s_a)$ is defined. 

We know by Relation I in \S~\ref{review} that 
\bgl
\label{axc=ac2}
  \Ad(s_a) \circ \Ad(s_c) = \Ad(s_{ac}).
\egl
So, using \eqref{axc=ac1}, we can deduce from \eqref{axc=ac2} that 
\bgl
\label{axc=ac}
  \vartheta_{ac,c}^{-1} \circ \rho_{c,a} \circ \vartheta_{c,1}^{-1} \circ \rho_{1,c} = \vartheta_{ac,1}^{-1} \circ \rho_{1,ac}.
\egl

Moreover, $\Ad(s_a) \circ \iota_{1,c} = \iota_{a,ac} \circ \Ad(s_a)$ and \eqref{axc=ac1} imply
\bgl
\label{iotatheta}
  \vartheta_{ac,c}^{-1} \circ \rho_{c,a} \circ \iota_{1,c} = \iota_{a,ac} \circ \vartheta_{a,1}^{-1} \circ \rho_{1,a}.
\egl

Finally, we compute 
\bglnoz
  && (\rho_{1,ac})_*^{-1} (\vartheta_{ac,1})_* (\iota_{a,ac})_* (\vartheta_{a,1}^{-1})_* (\rho_{1,a})_* \\
  &\overset{\eqref{axc=ac}}{=}& 
  (\rho_{1,c})_*^{-1} (\vartheta_{c,1})_* (\rho_{c,a})_*^{-1} (\vartheta_{ac,c})_* (\iota_{a,ac})_* (\vartheta_{a,1}^{-1})_* (\rho_{1,a})_* \\
  &\overset{\eqref{iotatheta}}{=}& (\rho_{1,c})_*^{-1} (\vartheta_{c,1})_* (\rho_{c,a})_*^{-1} (\vartheta_{ac,c})_* 
  (\vartheta_{ac,c}^{-1})_* (\rho_{c,a})_* (\iota_{1,c})_* \\
  &=& (\rho_{1,c})_*^{-1} (\vartheta_{c,1})_* (\iota_{1,c})_*
\eglnoz
\eproof

Therefore it remains to determine 
\bgloz
  (\rho_{1,c})_*^{-1} (\vartheta_{c,1})_* (\iota_{1,c})_*: K_0(C^*(\gekl{u^b}, s_{\zeta})) \ri K_0(C^*(\gekl{u^b}, s_{\zeta})).
\egloz
Let us denote this map by $\eta_c$, i.e., $\eta_c = (\rho_{1,c})_*^{-1} (\vartheta_{c,1})_* (\iota_{1,c})_*$. In conclusion, we have identified the K-theory of $C^*(\gekl{u^b},s_{\zeta},\gekl{e_a})$ with $\ilim_c(K_0(C^*(\gekl{u^b}, s_{\zeta})),\eta_c)$.

\subsection{The structure maps}

First of all, we can canonically identify $C^*(\gekl{u^b},s_{\zeta})$ with $C^*(R \rtimes \mu)$ because $R \rtimes \mu$ is amenable. Therefore, all the results from \S~\ref{K(group-C)} carry over to $C^*(\gekl{u^b},s_{\zeta})$. In the sequel, we use the same notations as in \S~\ref{K(group-C)}, but everything should be understood modulo this canonical isomorphism $C^*(\gekl{u^b},s_{\zeta}) \cong C^*(R \rtimes \mu)$.

To determine $\eta_c$, we use the decomposition $K_0(C^*(\gekl{u^b},s_{\zeta})) = K_{inf} \oplus K_{fin}^{\co} \oplus K_{fin}^\mu$ with the particular $\Zz$-basis $\gekl{\overline{y}_1, \dotsc, \overline{y}_{\rk_{fin}^{\co}}}$ and $\gekl{\overline{z}_1, \dotsc, \overline{z}_{m-1}}$ of $K_{fin}^{\co}$ and $K_{fin}^\mu$, respectively (see \S~{K(group-C)}). Moreover, let $\eckl{1} \in K_0(C^*(\gekl{u^b},s_{\zeta}))$ be the $K_0$-class of the unit in $C^*(\gekl{u^b},s_{\zeta})$ and denote by $\spkl{\eckl{1}}$ the subgroup of $K_0(C^*(\gekl{u^b},s_{\zeta}))$ generated by $\eckl{1}$. As the canonical inclusion $\Cz \cdot 1 \into C^*(\gekl{u^b},s_{\zeta})$ splits (a split is given by $C^*(\gekl{u^b},s_{\zeta}) \cong C^*(R \rtimes \mu) \to C^*(\gekl{e}) \cong \Cz \cdot 1$), it is clear that $\spkl{\eckl{1}}$ is a direct summand of $K_0(C^*(\gekl{u^b},s_{\zeta}))$, hence of $K_{inf}$.

Furthermore, note that it suffices to determine the structure maps $\eta_c$ for $c$ in $\Zz_{>1}$ with the property that $\prod_{i \tei m, 1 \leq i<m} (1-\zeta^i)$ divides $c$ because these elements form a cofinal set in $R \reg$ with respect to divisibility.

Our goal is to prove
\bprop
\label{etac}
There exists a subgroup $K_{inf}^{\co}$ of $K_{inf}$ together with a $\Zz$-basis $\gekl{\overline{x}_1, \dotsc, \overline{x}_{\rk_{inf}^{\co}}}$ of $K_{inf}^{\co}$ such that $K_{inf} = \spkl{\eckl{1}} \oplus K_{inf}^{\co}$ and that with respect to the $\Zz$-basis $\gekl{\eckl{1}, \overline{x}_1, \dotsc, \overline{x}_{\rk_{inf}^{\co}}, \overline{y}_1, \dotsc, \overline{y}_{\rk_{fin}^{\co}}, \overline{z}_1, \dotsc, \overline{z}_{m-1}}$ of $K_0(C^*(\gekl{u^b},s_{\zeta}))$, for every $c$ in $\Zz_{>1}$ with the property that $\prod_{i \tei m, 1 \leq i<m} (1-\zeta^i)$ divides $c$, the homomorphism $\eta_c$ is of the form
\bgloz
  \rukl{
  \begin{array}{c|ccc|c|ccc}
  c^n & & * & & * & & * & \\
  \hline
   & \ddots & & * & & & & \\
  0 & & \ddots & & * & & * & \\
   & 0 & & c^? & & & & \\
  \hline
  0 & & 0 & & 0 & & * & \\
  \hline
   & & & & & 1 & & * \\
  0 & & 0 & & 0 & & \ddots & \\
   & & & & & 0 & & 1 \\
  \end{array}
  }.
\egloz
This matrix is subdivided according to the decomposition $K_0(C^*(\gekl{u^b},s_{\zeta})) = \spkl{\eckl{1}} \oplus K_{inf}^{\co} \oplus K_{fin}^{\co} \oplus K_{fin}^\mu$. Moreover, the diagonal of the box
$
  \begin{array}{|ccc|}
  \hline
  \ddots & & * \\
   & \ddots & \\
  0 & & c^? \\
  \hline
  \end{array}
$
describing the $K_{inf}^{\co}$-$K_{inf}^{\co}$-part of this matrix consists of powers of $c$ with decreasing exponents. The least exponent $?$ can be $0$ only if $n$ is even, and in that case, the $0$-th power $c^0$ can appear only once on the diagonal.
\eprop
The proof of this proposition consists of two parts which are treated in the following two paragraphs.

\subsubsection{The infinite part}
\blemma
\label{etac_Kinf}
For $c \in \Zz_{>1}$, we have $\eta_c(K_{inf}) \subseteq K_{inf}$. Moreover, there is a subgroup $K_{inf}^{\co}$ of $K_{inf}$ and a $\Zz$-basis $\gekl{\overline{x}_1, \dotsc, \overline{x}_{\rk_{inf}^{\co}}}$ of $K_{inf}^{\co}$ such that $K_{inf} = \spkl{\eckl{1}} \oplus K_{inf}^{\co}$ and, for every $c \in \Zz_{>1}$, $\eta_c \vert_{K_{inf}}$, as a map $K_{inf} \to K_{inf}$, is of the form
\bgloz
  \rukl{
  \begin{array}{c|ccc}
   c^n & & * & \\
  \hline
   & \ddots & & * \\
   0 & & \ddots & \\
   & 0 & & c^?
\end{array}
}
\egloz
with respect to the decomposition $K_{inf} = \spkl{\eckl{1}} \oplus K_{inf}^{\co}$ and the chosen $\Zz$-basis of $K_{inf}^{\co}$. Here, as in the proposition, $?$ can be $0$ only if $n$ is even, and in that case, the $0$-th power $c^0$ can only appear at most once on the diagonal. 
\elemma
\bproof
Let us choose a suitable $\Zz$-basis for $K_{inf}$ and determine $\eta_c \vert_{K_{inf}}$. First of all, under the canonical identification $C^*(\gekl{u^b},s_{\zeta}) \cong C^*(R \rtimes \mu)$, the sub-C*-algebra $C^*(\gekl{u^b})$ corresponds to $C^*(R)$. So the inclusion map $\iota: C^*(R) \into C^*(R \rtimes \mu)$ corresponds to the canonical inclusion $C^*(\gekl{u^b}) \into C^*(\gekl{u^b},s_{\zeta})$ which we denote by $\iota$ as well. Let $\omega_1, \dotsc, \omega_n$ be a $\Zz$-basis for $R$ and let $u(i) \defeq u^{\omega_i}$. Since $C^*(\gekl{u^b})$ is isomorphic to $C^*(R) \cong C^*(\Zz^n)$ ($R$ is viewed as an additive group), a $\Zz$-basis for $K_0(C^*(\gekl{u^b}))$ is given by
\bgloz
  \menge{\eckl{u(i_1)} \times \dotsm \times \eckl{u(i_k)}}{i_1 < \dotsb < i_k, k \text{ even}}.
\egloz
Here $\times$ is the exterior product in K-theory as described in \cite{Hig-Roe}. Moreover, $\eckl{\cdot}$ denotes the $K_1$-class of the unitary in question.

Let $\nu_c$ be the endomorphism on $C^*(\gekl{u^b})$ defined by $\nu_c(u^b) = u^{cb}$. We have 
\bgl
\label{thetaiotanu}
  (\vartheta_{c,1} \circ \iota_{1,c} \circ \iota \circ \nu_c) (u^b) = \vartheta_{c,1} (u^{cb}) = 1 \otimes u^b
\egl 
for all $b$ in $R$. Thus $\vartheta_{c,1} \circ \iota_{1,c} \circ \iota \circ \nu_c = (1 \otimes \id) \circ \iota$. We conclude that 
\bgln
\label{etactorus}
  && \eta_c(\iota_*(\eckl{u(i_1)}_1 \times \dotsb \times \eckl{u(i_k)}_1)) \\ 
  &=& (\rho_{1,c})_*^{-1} (\vartheta_{c,1})_* (\iota_{1,c})_* (\iota_*(\eckl{u(i_1)}_1 \times \dotsb \times \eckl{u(i_k)}_1)) \nonumber \\ 
  &=& c^{-k} (\rho_{1,c})_*^{-1} (\vartheta_{c,1})_* (\iota_{1,c})_* (\iota_* (\nu_c)_*(\eckl{u(i_1)}_1 \times \dotsb \times \eckl{u(i_k)}_1)) \nonumber \\ 
  &\overset{\eqref{thetaiotanu}}{=}& c^{-k} (\rho_{1,c})_*^{-1} (1 \otimes \id)_* (\iota_*(\eckl{u(i_1)}_1 \times \dotsb \times \eckl{u(i_k)}_1)) \nonumber \\ 
  &=& c^{n-k} \iota_*(\eckl{u(i_1)}_1 \times \dotsb \times \eckl{u(i_k)}_1). \nonumber
\egln
Now, let $H_k$ be the subgroup of $K_0(C^*(\gekl{u^b}))$ generated by the $K_0$-classes $\eckl{u(i_1)} \times \dotsm \times \eckl{u(i_k)}$ for $i_1 < \dotsb < i_k$ where $k$ is fixed. We have
\bgloz
  K_0(C^*(\gekl{u^b})) = \bigoplus_{k \geq 0 \text{ even}} H_k.
\egloz
We claim that $\ker(\iota_*)$ is compatible with this decomposition, i.e.,
\bgloz
  \ker(\iota_*) = \bigoplus_{k \geq 0 \text{ even}} (H_k \cap \ker(\iota_*)).
\egloz
Proof of the claim: 

Let $h$ be in $\ker(\iota_*)$. We can write
\bgl
\label{h}
  h = \sum_{k \geq 0 \text{ even}} h_k
\egl
with $h_k \in H_k$. We have to show that for every $k$, the summand $h_k$ lies in $\ker(\iota_*)$. Let us assume that there are at least two non-zero summands in \eqref{h}, because otherwise, there is nothing to show. Now, equation~\eqref{etactorus} tells us that
\bgl
\label{etaiota}
  \eta_c \circ \iota_* = \iota_* \circ (\bigoplus_k (c^{n-k} \cdot \id_{H_k}))
\egl
on $K_0(C^*(\gekl{u^b})) = \bigoplus_k H_k$. Thus $(\bigoplus_k (c^{n-k} \cdot \id_{H_k}))(h) = \sum_k c^{n-k} \cdot h_k$ lies in $\ker(\iota_*)$ as well. This implies
\bgloz
  \ker(\iota_*) \ni c^n h - (\bigoplus_k (c^{n-k} \cdot \id_{H_k}))(h) = \sum_{k \geq 2 \text{ even}} (c^n - c^{n-k}) \cdot h_k.
\egloz
Proceeding inductively, we obtain that for every even number $j \geq 2$,
\bgl
\label{diffh}
  \sum_{k \geq j \text{ even}} (c^n - c^{n-k})(c^{n-2} - c^{n-k}) \dotsm (c^{n-j+2} - c^{n-k}) \cdot h_k
\egl
lies in $\ker(\iota_*)$. Taking $j$ to be the highest index for which the summand $h_j$ in \eqref{h} is not zero, the term in \eqref{diffh} will be a non-zero multiple of the highest term in \eqref{h}. As both $K_0(C^*(\gekl{u^b}))$ and $K_0(C^*(\gekl{u^b},s_{\zeta}))$ are free abelian, we conclude that the highest term itself must lie in $\ker(\iota_*)$. Working backwards, we obtain that for every $k$, the summand $h_k$ lies in $\ker(\iota_*)$. This proves our claim.

Now, for every $k$, $H_k \cap \ker(\iota_*) = \ker(\iota_* \vert_{H_k})$ is a direct summand of $H_k$ because $K_0(C^*(\gekl{u^b},s_{\zeta}))$ is free abelian. Thus we can choose subgroups $I_k$ of $H_k$ so that
\bgloz
  H_k = I_k \oplus (H_k \cap \ker(\iota_*)).
\egloz
As $\ker(\iota_*) = \bigoplus_k (H_k \cap \ker(\iota_*))$, we have $K_0(C^*(\gekl{u^b})) = (\bigoplus_k I_k) \oplus \ker(\iota_*)$. We can choose a $\Zz$-basis for $\bigoplus_k I_k$ in $\bigcup_k H_k$. As $H_0 \cap \ker(\iota_*) = \gekl{0}$, we have $I_0 = H_0 = \spkl{\eckl{1}}$ so that we can let $\eckl{1}$ be a basis element. Moreover, $H_n$ is non-trivial only if $n$ is even, and in that case $\rk(H_n)=1$ so that there is at most one basis element in $H_n$.

Now $\iota_*$ maps $\bigoplus_k I_k$ isomorphically into $\img(\iota_*) \subseteq K_0(C^*(\gekl{u^b},s_{\zeta}))$, so that the $\Zz$-basis of $\bigoplus_k I_k$ chosen above is mapped to a $\Zz$-basis of $\img(\iota_*)$. By \eqref{etaiota}, we know that if we order this $\Zz$-basis in the right way (corresponding to the index $k$), we obtain that $\eta_c \vert_{\img(\iota_*)}$ -- as an endomorphism of $\img(\iota_*)$ -- is given by
\bgloz
  \rukl{
  \begin{array}{ccc}
   c^n & & 0 \\
   & \ddots & \\
   0 & & c^?
\end{array}
}
\egloz
where the exponents of $c$ on the diagonal are monotonously decreasing. The entry $c^n$ corresponds to the basis element $\eckl{1}$, and $c^0$ can only appear at most once on the diagonal (if it appears, it has to be in the lower right corner according to our ordering). We enumerate this $\Zz$-basis of $\img(\iota_*)$ by $x_0, \dotsc, x_{\rk_{inf}^{\co}}$ according to our ordering, so $x_0=\eckl{1}$.

By construction (see \eqref{defKinf}), $\spkl{x_0, \dotsc, x_{\rk_{inf}^{\co}}} = \img(\iota_*)$ is of finite index in $K_{inf}$, so that we can choose a $\Zz$-basis $\overline{x}_0, \dotsc, \overline{x}_{\rk_{inf}^{\co}}$ of $K_{inf}$ with the property that
\bgloz
  \spkl{\overline{x}_0,\dotsc,\overline{x}_j} = \menge{x \in K_{inf}}{\text{There exists } N \in \Zz \pos \text{ with } Nx \in \spkl{\overline{x}_0,\dotsc,\overline{x}_j}}
\egloz
for every $0 \leq j \leq \rk(K_{inf})-1 = \rk_{inf}^{\co}$. In particular, we have $\overline{x}_0 = \eckl{1}$. It then follows that $\eta_c(K_{inf}) \subseteq K_{inf}$ and that with respect to the $\Zz$-basis $\gekl{\overline{x_j}}$, the matrix describing $\eta_c \vert_{K_{inf}}$ -- as an endomorphism of $K_{inf}$ -- is of the form
\bgloz
  \rukl{
  \begin{array}{ccc}
   c^n & & * \\
   & \ddots & \\
   0 & & c^?
\end{array}
}.
\egloz
Recall that $?$ can be $0$ only if $n$ is even, and in that case, the $0$-th power $c^0$ can only appear at most once on the diagonal. Now set $K_{inf}^{\co} \defeq \spkl{\overline{x}_1,\overline{x}_2,\dotsc,\overline{x}_{\rk_{inf}^{\co}}}$. Then we have $K_{inf} = \spkl{\eckl{1}} \oplus K_{inf}^{\co}$ by construction, and $\overline{x}_1,\overline{x}_2,\dotsc,\overline{x}_{\rk_{inf}^{\co}}$ is a $\Zz$-basis of $K_{inf}^{\co}$ with the desired properties.
\eproof
This first lemma settles the $K_{inf}$-part.

Up to now, we have only used that $c$ is an integer bigger than 1, the extra condition that $\prod_{i \tei m, 1 \leq i<m} (1-\zeta^i)$ divides $c$ was not used in our arguments up to this point. But for the finite part, this condition plays a crucial role.

\subsubsection{The finite part}

\blemma
\label{etac_Kfin}
Assume that $c$ is an integer bigger than 1 and that $\prod_{i \tei m, 1 \leq i<m} (1-\zeta^i)$ divides $c$. Then
\bgloz
  \eta_c(\eckl{p_\chi(s_\zeta)}) \in \eckl{p_\chi(s_\zeta)} + \sum_{(\mu) \neq (M) \in \cM} \img((\iota_M)_*) \fa \chi \in \widehat{\Zz/m\Zz}
\egloz
and
\bgloz
  \eta_c\rukl{\sum_{(\mu) \neq (M) \in \cM} \img((\iota_M)_*)} \in \spkl{\eckl{1}}.
\egloz
\elemma
\bproof
Let $M$ be a maximal finite subgroup of $R \rtimes \mu$, and choose a generator $(b,\zeta^i) \in R \rtimes \mu$ of $M$. Our aim is to compute $\eta_c(\eckl{p_\chi(u^b s_{\zeta^i})})$ (with $\chi \in \widehat{\Zz/\tfrac{m}{i}\Zz}$). By definition, $\eta_c = (\rho_{1,c})_*^{-1} (\vartheta_{c,1})_* (\iota_{1,c})_*$. So we have to examine $(\vartheta_{c,1} \circ \iota_{1,c}) (u^b s_{\zeta^i})$. Take two elements $d$, $d'$ in the system $\cR_c$ of representatives for $R/cR$. The $(d,d')$-th entry of $(\vartheta_{c,1} \circ \iota_{1,c}) (u^b s_{\zeta^i})$ is given by
\bglnoz
  s_c^* u^{-d} u^b s_{\zeta^i} u^{d'} s_c &=& s_c^* u^{-d+b+\zeta^i d'} s_c s_{\zeta^i} \\
  &=&
  \bfa
    0 \falls -d+b+\zeta^i d' \notin cR, \\
    u^{c^{-1}(-d+b+\zeta^i d')} s_{\zeta^i} \falls -d+b+\zeta^i d' \in cR.
  \efa
\eglnoz
Therefore, the matrix $(\vartheta_{c,1} \circ \iota_{1,c}) (u^b s_{\zeta^i})$ has exactly one non-zero entry in each row and column. In other words, for fixed $d$ in $\cR_c$, there exists exactly one $d' \in \cR_c$ with $-d+b+\zeta^i d' \in cR$, namely the element in $\cR_c$ which represents the coset $\zeta^{-i}(d-b)+cR$. There is only one such element because $\cR_c$ is minimal.

Moreover, $u^b s_{\zeta^i}$ is a cyclic element: Its $\tfrac{m}{i}$-th power is $1$. 

These two observations imply that $\tfrac{1}{m/i} \sum_{j=0}^{\tfrac{m}{i}-1} \chi(j+\tfrac{m}{i}\Zz) (u^b s_{\zeta^i})^j = p_{\chi}(u^b s_{\zeta^i})$ can be decomposed into irreducible summands, and each of these summands has to be a projection. \an{Irreducible} means that once we apply $\vartheta_{c,1} \circ \iota_{1,c}$, we obtain an irreducible matrix. Thus up to conjugation by a permutation matrix, $(\vartheta_{c,1} \circ \iota_{1,c}) (p_{\chi}(u^b s_{\zeta^i}))$ is of the form
\bgloz
  \rukl{
  \begin{array}{ccc}
  p_1 & & 0 \\
  & p_2 & \\
  0 & & \ddots
  \end{array}
  }
\egloz
where the $p_i$ are projections of certain sizes. Of course, conjugation by a permutation matrix does not have any effect in K-theory. This means that we obtain
\bgloz
  \eta_c(\eckl{p_{\chi}(u^b s_{\zeta^i})}) = (\rho_{1,c})_*^{-1} (\eckl{p_1} + \eckl{p_2} + \dotsb)
\egloz
where $\rho_{1,c}$ is the homomorphism 
\bgloz
  C^*(\gekl{u^b},s_{\zeta}) \ri \cL(\ell^2(R/cR)) \otimes C^*(\gekl{u^b},s_{\zeta}); \ x \ma e_{0,0} \otimes x.
\egloz
Here $\cL(\ell^2(R/cR)) \cong M_{c^n}(\Cz)$ and $e_{0,0}$ is a minimal projection. So it remains to find out what these irreducible summands $p_i$ give in K-theory.

First of all, we look at the case $b=0$, $i=1$, i.e., we consider $p_{\chi}(s_{\zeta})$. Irreducible summands of size 1 must be of the form $p_{\chi}(u^b s_{\zeta})$. What we want to show now is that there is only one 1-dimensional summand which gives the class of $p_{\chi}(s_{\zeta})$. To do so, we take a 1-dimensional summand corresponding to the position $d$ (for some $d$ in $\cR_c$). The $(d,d)$-th entry of $(\vartheta_{c,1} \circ \iota_{1,c}) (s_{\zeta})$ is given by $u^{c^{-1}(\zeta d-d)} s_{\zeta}$. By Theorem~\ref{LL} the corresponding projection (i.e., $p_{\chi}(u^{c^{-1}(\zeta d-d)} s_{\zeta})$) gives a $K_0$-class in $\img((\iota_\mu)_*)$ if and only if the subgroups $\spkl{(c^{-1}(\zeta d-d),\zeta)} \text{ and } \spkl{(0,\zeta)}$ of $R \rtimes \mu$ are conjugate. This is equivalent to $c^{-1}(\zeta d-d) \in (\zeta-1) \LRarr d \in cR$. But as $\cR_c$ is minimal, this happens for exactly one element in $\cR_c$ (by our convention, this element has to be $0$, but this is not important at this point). Moreover, if $d$ lies in $cR$, then $\eckl{p_\chi(u^{c^{-1}(\zeta d-d)} s_{\zeta})} = \eckl{p_\chi(s_{\zeta})}$ in $K_0$. So from the 1-dimensional summands we obtain in $K_0$ exactly once the class $\eckl{p_{\chi}(s_{\zeta})}$ and some other classes in $\sum_{(\mu) \neq (M) \in \cM} \img((\iota_M)_*)$.

It remains to examine higher dimensional summands. We want to show that all the higher dimensional summands give rise to $K_0$-classes in $\sum_{(\mu) \neq (M) \in \cM} \img((\iota_M)_*)$. Let us take a summand of size $j$ with $j>1$. This means that the $j$-th power of $(\vartheta_{c,1} \circ \iota_{1,c}) (s_{\zeta})$ has a non-zero diagonal entry, say at the $(d,d)$-th position. This entry is $u^{c^{-1}(\zeta^j d-d)} s_{\zeta^j}$. Now we prove a result in a bit more generality than actually needed at this point. But later on, we will come back to it.

\blemma
\label{MvN}
Assume that for an irreducible summand of $(\vartheta_{c,1} \circ \iota_{1,c}) (p_{\chi}(u^b s_{\zeta^i}))$, $j \in \Zz_{>1}$ is the smallest number such that the $j$-th power of this summand has non-zero diagonal entries. Let one of these non-zero diagonal entries be $u^{\ti{b}} s_{\zeta^{ij}}$ for some $\ti{b}$ in $R$. Then the $K_0$-class of this summand coincides with $\eckl{p_{\ti{\chi}}(u^{\ti{b}} s_{\zeta^{ij}})}$ where $\ti{\chi}$ is the restriction of $\chi \in \widehat{\spkl{\zeta^i}}$ to $\widehat{\spkl{\zeta^{ij}}}$.
\elemma
\bproof[Proof of Lemma~\ref{MvN}]
Up to conjugation by a permutation matrix, the irreducible summand of $u^b s_{\zeta^i}$ we are considering is of the form
\bgloz
  \rukl{
  \begin{array}{cccc}
  0 & & & x_j \\
  x_1 & \ddots & & \\
  & \ddots & \ddots & \\
  0 & & x_{j-1} & 0
  \end{array}
  }.
\egloz
All the entries lie in $C^*(\gekl{u^b},s_{\zeta})$.

The $j$-th power is given by
\bgloz
  \rukl{
  \begin{array}{ccccc}
  x_j x_{j-1} \dotsm x_2 x_1 & & & 0 \\
  0 & x_1 x_j x_{j-1} \dotsm x_2 & & 0 \\
  & & \ddots & \\
  0 & & & x_{j-1} x_{j-2} \dotsm x_1 x_j
  \end{array}
  }.
\egloz
By assumption, $x_j x_{j-1} \dotsm x_2 x_1 = u^{\ti{b}} s_{\zeta^{ij}}$. Then the irreducible summand of 
\bgloz
  p_{\chi}(u^b s_{\zeta^i}) = \tfrac{1}{m/i} \sum_{k=0}^{\tfrac{m}{i}-1} \chi(j) (u^b s_{\zeta^i})^k
\egloz
is given by
\bgloz
  \tfrac{1}{j}
  \rukl{
  \begin{array}{ccc}
  p_{\ti{\chi}}(u^{\ti{b}} s_{\zeta^{ij}}) & p_{\ti{\chi}}(u^{\ti{b}} s_{\zeta^{ij}}) \cdot x_1^* & \dotso \\
  x_1 \cdot p_{\ti{\chi}}(u^{\ti{b}} s_{\zeta^{ij}}) & x_1 \cdot p_{\ti{\chi}}(u^{\ti{b}} s_{\zeta^{ij}}) \cdot x_1^* & \dotso \\
  \vdots & \vdots & \vdots \\
  x_{j-1} \dotsm x_1 \cdot p_{\ti{\chi}}(u^{\ti{b}} s_{\zeta^{ij}}) & x_{j-1} \dotsm x_1 \cdot p_{\ti{\chi}}(u^{\ti{b}} s_{\zeta^{ij}}) \cdot x_1^* & \dotso
  \end{array}
  }.
\egloz
The $k$-th column is given by the product of the first column with $(x_{k-1} \dotsm x_1)^*$ from the right. 

But then,
\bgloz
  \tfrac{1}{\sqrt{j}}
  \rukl{
  \begin{array}{ccc}
  p_{\ti{\chi}}(u^{\ti{b}} s_{\zeta^{ij}}) & 0 & \dotso \\
  \vdots & \vdots & 0 \\
  x_{j-1} \dotsm x_1 \cdot p_{\ti{\chi}}(u^{\ti{b}} s_{\zeta^{ij}}) & 0 & \dotso
  \end{array}
  }
\egloz
is a partial isometry with entries in $C^*(\gekl{u^b},s_{\zeta})$ whose range projection is precisely the irreducible summand from above and whose support projection is
\bgloz
  \rukl{
  \begin{array}{cccc}
  p_{\ti{\chi}}(u^{\ti{b}} s_{\zeta^{ij}}) & 0 & \dotso & 0 \\
  0 & & & \\
  \vdots & & 0 & \\
  0 & & &
  \end{array}
  }.
\egloz
This proves Lemma~\ref{MvN}.
\eproof
\bcor
\label{corMvN}
If in the situation of Lemma~\ref{MvN}, we have $j = \tfrac{m}{i}$, i.e., $u^{\ti{b}} s_{\zeta^{ij}} = 1$, then the corresponding irreducible summand of $(\vartheta_{c,1} \circ \iota_{1,c}) (p_{\chi}(u^b s_{\zeta^i}))$ gives the $K_0$-class $\eckl{1}$.
\ecor

Now let us continue the proof of Lemma~\ref{etac_Kfin}. We go back to the higher dimensional summands of $(\vartheta_{c,1} \circ \iota_{1,c}) (s_{\zeta})$. We were considering the $(d,d)$-th position with entry $u^{c^{-1}(\zeta^j d-d)} s_{\zeta^j}$. Lemma~\ref{MvN} tells us that this irreducible summand gives $\eckl{p_{\ti{\chi}}(u^{c^{-1}(\zeta^j d-d)} s_{\zeta^j})}$. By Theorem~\ref{LL}, this $K_0$-class lies in $\img((\iota_\mu)_*)$ if and only if the corresponding subgroup $\spkl{(c^{-1}(\zeta^j d-d),\zeta^j)}$ is conjugate to a subgroup of $\spkl{(0,\zeta)}$. In case $\zeta^j \neq 1$, this happens if and only if $c^{-1}(\zeta^j d-d) \in (\zeta^j -1) \LRarr d \in cR \LRarr d = 0$ by our choice of $\cR_c$. But for $d=0$, the summand we get is of size $1$ (see above). This contradicts $j>1$. If $\zeta^j = 1$, then we obtain a projection whose $K_0$-class is $\eckl{1}$ by Corollary~\ref{corMvN}.

This proves $\eta_c(\eckl{p_\chi(s_\zeta)}) \in \eckl{p_\chi(s_\zeta)} + \sum_{(\mu) \neq (M) \in \cM} \img((\iota_M)_*) \fa \chi \in \widehat{\Zz/m\Zz}$.

It remains to prove $\eta_c\rukl{\sum_{(\mu) \neq (M) \in \cM} \img((\iota_M)_*)} \in \spkl{\eckl{1}}$. Take a maximal finite subgroup $M$ with $(M) \neq (\mu)$. Let $(b,\zeta^i)$ be a generator of $M$, and consider the element $\eckl{p_{\chi}(u^b s_{\zeta^i})}$. How do the irreducible summands of $(\vartheta_{c,1} \circ \iota_{1,c})(p_{\chi}(u^b s_{\zeta^i}))$ look like? We claim that each of these summands must have size $\tfrac{m}{i}$. To show this, let $j$ be the size of such a summand. This means that there exists $d \in \cR_c$ such that
\bglnoz
  s_c^* u^{-d} (u^b s_{\zeta^i})^j u^d s_c &=& s_c^* u^{-d} u^{b+\zeta^i b+\dotsb+\zeta^{i(j-1)}b} s_{\zeta^{ij}} u^d s_c \\
  &=& s_c^* u^{-d+\tfrac{1-\zeta^{ij}}{1-\zeta^i}b+\zeta^{ij}d} s_c s_{\zeta^{ij}} \neq 0.
\eglnoz
This happens if and only if $-d+\tfrac{1-\zeta^{ij}}{1-\zeta^i}b+\zeta^{ij}d$ lies in $cR$. Now, if $\zeta^{ij} \neq 1$, then we can proceed as follows:
As $\prod_{i \tei m, 1 \leq i<m} (1-\zeta^i)$ divides $c$ by assumption, we know that $1-\zeta^{ij}$ divides $cR$, so that $-d+\tfrac{1-\zeta^{ij}}{1-\zeta^i}b+\zeta^{ij}d \in cR$ implies that $b$ lies in $(1-\zeta^i)R$. But this is a contradiction to $(M) \neq (\mu)$. Thus we must have $\zeta^{ij}=1$, which by minimality of $j$ implies $j=\tfrac{m}{i}$, as claimed.

Therefore all these irreducible summands give the $K_0$-class $\eckl{1}$ (see Corollary~\ref{corMvN}). This proves that 
\bgloz
  \eta_c\rukl{\sum_{(\mu) \neq (M) \in \cM} \img((\iota_M)_*)} \in \spkl{\eckl{1}}.
\egloz
\eproof

\bcor
With respect to the $\Zz$-basis $\overline{y}_1, \dotsc, \overline{y}_{\rk_{fin}^{\co}}$ and $\overline{z}_1, \dotsc, \overline{z}_{m-1}$ of $K_{fin}^{\co}$ and $K_{fin}^{\mu}$, respectively, and with respect to the $\Zz$-basis 
\bgloz
  \gekl{\eckl{1}, \overline{x}_1,\dotsc,\overline{x}_{\rk_{inf}^{\co}},\overline{y}_1,\dotsc,\overline{y}_{\rk_{fin}^{\co}},\overline{z}_1,\dotsc,\overline{z}_{m-1}}
\egloz
of $K_0(C^*(\gekl{u^b},s_{\zeta}))$, for every $c$ in $\Zz_{>1}$ with the property that $\prod_{i \tei m, 1 \leq i<m} (1-\zeta^i)$ divides $c$, we have that $\eta_c \vert_{K_{fin}^{\co} \oplus K_{fin}^{\mu}}: K_{fin}^{\co} \oplus K_{fin}^{\mu} \to K_0(C^*(\gekl{u^b},s_{\zeta}))$ is of the form
\bgloz
  \rukl{
  \begin{array}{c|ccc}
  * & & * & \\
  \hline
  * & & * & \\
  \hline
  0 & & * & \\
  \hline
   & 1 & & * \\
  0 & & \ddots & \\
   & 0 & & 1 \\
  \end{array}
  }.
\egloz
\ecor
\bproof
This follows from Lemma~\ref{etac_Kfin} and the way the basis elements $\overline{y}_1, \dotsc, \overline{y}_{\rk_{fin}^{\co}}$ and $\overline{z}_1, \dotsc, \overline{z}_{m-1}$ were chosen (see \S~\ref{K(group-C)}).
\eproof
With this corollary, together with Lemma~\ref{etac_Kinf}, we have completed the proof of Proposition~\ref{etac}.

\subsection{K-theory for a sub-C*-algebra}

Now we can compute the K-theory of $C^*(\gekl{u^b},s_{\zeta},\gekl{e_a})$. We can also determine $\Ad(s_c) = s_c \sqcup s_c^*$ on K-theory.

\bprop
\label{K-subalg}
We have
\bgloz
  K_0(C^*(\gekl{u^b},s_{\zeta},\gekl{e_a})) \cong \Qz^{\rk(K_{inf})-\delta} \oplus \Zz^\delta \oplus \Zz^{m-1}.
\egloz
Here $\delta$ is $1$ if for all $c$ in $\Zz_{>1}$ divisible by $\prod_{i \tei m, 1 \leq i<m} (1-\zeta^i)$, the least exponent of $c$ in the matrix describing $\eta_c$ (see Proposition~\ref{etac}) is $0$. Otherwise $\delta$ is $1$.

$K_1(C^*(\gekl{u^b},s_{\zeta},\gekl{e_a}))$ vanishes.

Moreover, there exists a $\Qz$-basis of $\Qz^{\rk(K_{inf})-\delta}$ and a $\Zz$-basis of $\Zz^\delta \oplus \Zz^{m-1}$ such that, for all $c \in \Zz_{>1}$ with $\prod_{i \tei m, 1 \leq i<m} (1-\zeta^i)$ dividing $c$, the homomorphism $\Ad(s_c)$ is of the form
\bgl
\label{Adsc-1}
  \rukl{
  \begin{array}{ccc|c|ccc}
  c^{-n} & & * & & & & \\
   & \ddots & & 0 & & 0 & \\
  0 & & \ddots & & & \\
  \hline
   & 0 & & 1 & & * & \\
  \hline  
  & & & & 1 & & * \\
  & 0 & & 0 & & \ddots & \\
  & & & & 0 & & 1
  \end{array}
  }
  \falls \delta = 1
\egl
and 
\bgl
\label{Adsc-0}
  \rukl{
  \begin{array}{ccc|ccc}
  c^{-n} & & * & & & \\
   & \ddots & & & 0 & \\
  0 & & \ddots & & \\
  \hline  
  & & & 1 & & * \\
  & 0 & & & \ddots & \\
  & & & 0 & & 1
  \end{array}
  }
  \falls \delta = 0.
\egl
\eprop
Note that in the first box on the diagonal of these matrices, all the diagonal entries are always strictly less than $1$.
\bproof
We know that $C^*(\gekl{u^b},s_{\zeta},\gekl{e_a})$ can be identified with the inductive limit obtained from the C*-algebras $C^*(\gekl{u^b},s_{\zeta},e_a) \text{, for } a \in R \reg$ and the inclusion maps $\iota_{a,ac}: C^*(\gekl{u^b},s_{\zeta},e_a) \ri C^*(\gekl{u^b},s_{\zeta},e_{ac})$. Using continuity of K-theory, together with Lemma~\ref{bd--d} and Lemma~\ref{a,ac->1,c}, we obtain for $i=0,1$:
\bgloz
  K_i(C^*(\gekl{u^b},s_{\zeta},\gekl{e_a})) \cong \ilim_c(K_i(C^*(\gekl{u^b},s_{\zeta})),\eta_c).
\egloz
From this, it is immediate that $K_1(C^*(\gekl{u^b},s_{\zeta},\gekl{e_a}))$ vanishes by ($*$$*$$*$) in Theorem~\ref{LL}. Moreover, the description for $K_0(C^*(\gekl{u^b},s_{\zeta},\gekl{e_a}))$ can be deduced from the description of $\eta_c$ in Proposition~\ref{etac}.

Concerning the description of $(\Ad(s_c))_*$, let us explain the case $\delta = 1$. The case $\delta = 0$ is similar. We observe that $(\Ad(s_c))_*$ is given by the inverse of the homomorphism on
\bgloz
  K_0(C^*(\gekl{u^b},s_{\zeta},\gekl{e_a})) \cong \ilim_c(K_0(C^*(\gekl{u^b},s_{\zeta})),\eta_c).
\egloz
induced by $\eta_c$. This follows from \eqref{axc=ac1}. We obtain that there exists a $\Qz$-basis of $\Qz^{\rk(K_{inf})-\delta}$ and a $\Zz$-basis of $\Zz^\delta \oplus \Zz^{m-1}$ such that $\Ad(s_c)$ is of the form
\bgl
\label{Adsc-1'}
  \rukl{
  \begin{array}{ccc|c|ccc}
  c^{-n} & & * & & & & \\
   & \ddots & & * & & * & \\
  0 & & \ddots & & & \\
  \hline
   & 0 & & 1 & & * & \\
  \hline  
  & & & & 1 & & * \\
  & 0 & & 0 & & \ddots & \\
  & & & & 0 & & 1
  \end{array}
  }.
\egl
Modifying the $\Zz$-basis if necessary, we can find a new $\Zz$-basis for $\Zz^\delta \oplus \Zz^{m-1}$ such that the two boxes in the upper right corner of \eqref{Adsc-1'} vanish.
\eproof

\subsection{Passing over to the infinite adele space}

We can now compute K-theory for certain crossed products involving the profinite completion of $R$. Using the duality theorem, we are then able to pass over to the infinite adele space.

\bcor
\label{K-fininf1}
We have
\bgln
\label{K_0(finite)}
  && K_0(C(\Rbar) \rtimes R \rtimes \mu) \cong \Qz^{\rk(K_{inf})-\delta} \oplus \Zz^\delta \oplus \Zz^{m-1} \\
\label{K_1(finite)}
  && K_1(C(\Rbar) \rtimes R \rtimes \mu) \cong \gekl{0} \\
\label{K_0(infinite)}
  && K_0(C_0(\Az_{\infty}) \rtimes K \rtimes \mu) \cong \Qz^{\rk(K_{inf})-\delta} \oplus \Zz^\delta \oplus \Zz^{m-1} \\
\label{K_1(infinite)}
  && K_1(C_0(\Az_{\infty}) \rtimes K \rtimes \mu) \cong \gekl{0}.
\egln
\ecor
\bproof
These results follow from the duality theorem (see \S~\ref{review}) and Proposition~\ref{K-subalg}.
\eproof
For the next result, we need Proposition~\ref{K_0-inj}.
\bcor
\label{K-subalg'}
With respect to the same bases as in Proposition~\ref{K-subalg}, we must have that $\Ad(s_c)$ is of the form
\bgloz
  \rukl{
  \begin{array}{ccc|c}
  c^{-n} & & * & \\
   & \ddots & & 0 \\
  0 & & \ddots & \\
  \hline
   & & & \\
   & 0 & & \id_{\Zz^\delta \oplus \Zz^{m-1}} \\
   & & &
  \end{array}
  }
\egloz
for all $c \in \Zz_{>1}$ divisible by $\prod_{i \tei m, 1 \leq i<m} (1-\zeta^i)$.

Moreover, if $n$ is even, $\delta$ must be $1$.
\ecor
\bproof
If $n=[K:\Qz]$ is odd, $\delta$ must vanish and $m$ must be $2$ as $K$ admits an embedding into $\Rz$ so that $\mu$ must be $\gekl{\pm 1}$. So in that case, there is nothing to prove.

Now let us consider the case that $n=[K:\Qz]$ is even. First of all, we know that under the canonical isomorphism $C(\Rbar) \rtimes R \rtimes \mu \cong C^*(\gekl{u^b},s_\zeta,\gekl{e_a})$, the endomorphism $\beta^{(fin)}_c$ of $C(\Rbar) \rtimes R \rtimes \mu$ induced by multiplication with $c$ corresponds to $\Ad(s_c)$. From Proposition~\ref{K-subalg}, we see that $\rk(\ker(\id-(\Ad(s_c)_*))) \leq m$ and that we can only have equality if $\delta = 1$ and $\Ad(s_c)$ is the identity on $\Zz^\delta \oplus \Zz^{m-1}$ in Proposition~\ref{K-subalg}. Thus the same holds for $\rk(\ker(\id-(\beta_c^{(fin)})_*))$.

As against that, we know that $C_0(\Az_{\infty}) \rtimes K \sim_M C(\Rbar) \rtimes R$ in a $R\reg$-equivariant way (see \S~\ref{review}). Therefore, we can identify $K_*(C_0(\Az_\infty) \rtimes K \rtimes \mu)$ with $K_*(C(\Rbar) \rtimes R \rtimes \mu)$ so that $(\beta_c)_*$ ($\beta_c$ is the automorphism of $C_0(\Az_\infty) \rtimes K \rtimes \mu$ induced by multiplication with $c$) corresponds to $(\beta^{(fin)}_c)_*$. Thus $\rk(\ker(\id-(\beta_c^{(fin)})_*)) = \rk(\ker(\id-(\beta_c)_*))$. But the canonical homomorphism $C_0(\Az_{\infty}) \rtimes \mu \to C_0(\Az_{\infty}) \rtimes K \rtimes \mu$ maps into $\ker(\id-(\beta_c)_*)$ in $K_0$ since $\beta_c$ is homotopic to the identity on $C_0(\Az_{\infty}) \rtimes \mu$ (recall that $n$ is even). Moreover, by Proposition~\ref{K_0-inj}, we know that this canonical homomorphism is injective on $K_0$. As $K_0(C_0(\Az_{\infty}) \rtimes \mu) \cong K_0(C^*(\mu)) \cong \Zz^m$ by equivariant Bott periodicity (see \cite[Theorem~20.3.2]{Bla}, $n$ is even), we conclude that $\rk(\ker(\id-(\beta_c)_*)) \geq m$. So $\rk(\ker(\id-(\beta_c^{(fin)})_*))$ must be $m$, and our assertion follows.
\eproof
\bcor
If $K$ has higher roots of unity ($\abs{\mu} > 2$), then
\bgloz
  K_0(C_0(\Az_\infty) \rtimes K \rtimes \mu) \cong \Qz^{\rk(K_{inf})-1} \oplus \Zz^m \text{ and } K_1(C_0(\Az_\infty) \rtimes K \rtimes \mu \cong \gekl{0}.
\egloz
Moreover, there exists a $\Qz$-basis for $\Qz^{\rk(K_{inf})-1}$ and a $\Zz$-basis for $\Zz^m$ such that for $c$ in $\Zz_{>1}$ divisible by $\prod_{i \tei m, 1 \leq i<m} (1-\zeta^i)$, $(\beta_c)_*: K_0(C_0(\Az_\infty) \rtimes K \rtimes \mu) \to K_0(C_0(\Az_\infty) \rtimes K \rtimes \mu)$ is of the form
\bgl
\label{betac}
  \rukl{
  \begin{array}{ccc|c}
  c^{-n} & & * & \\
   & \ddots & & 0 \\
  0 & & \ddots & \\
  \hline
   & & & \\
   & 0 & & \id_{\Zz^m} \\
   & & &
  \end{array}
  }
\egl
\ecor
\bproof
$\abs{\mu} > 2$ implies that $n$ is even. So the second statement in Corollary~\ref{K-subalg'} tells us that $\delta = 1$. Thus our first two statements about $K_0$ and $K_1$ follow from \eqref{K_0(infinite)} and \eqref{K_1(infinite)}. Furthermore, we can identify $K_*(C_0(\Az_\infty) \rtimes K \rtimes \mu)$ with $K_*(C(\Rbar) \rtimes R \rtimes \mu)$ so that $(\beta_c)_*$ corresponds to $(\beta^{(fin)}_c)_*$. We can also identify $C(\Rbar) \rtimes R \rtimes \mu$ with $C^*(\gekl{u^b},s_\zeta,\gekl{e_a})$ so that $\beta^{(fin)}_c$ corresponds to $\Ad(s_c)$. Thus our second statement about $(\beta_c)_*$ follows from Corollary~\ref{K-subalg'}.
\eproof
\bcor
\label{K(inf-with-c)}
If $K$ has higher roots of unity ($\abs{\mu} > 2$), then we have for all $c$ in $\Zz_{>1}$ divisible by $\prod_{i \tei m, 1 \leq i<m} (1-\zeta^i)$ that $K_i(C_0(\Az_\infty) \rtimes K \rtimes (\mu \times \spkl{c})) \cong \Zz^m$ ($i=0,1$).
\ecor
\bproof
Just plug in the results from the last corollary into the Pimsner-Voiculescu sequence for $C_0(\Az_\infty) \rtimes K \rtimes (\mu \times \spkl{c}) \cong (C_0(\Az_\infty) \rtimes K \rtimes \mu) \rtimes_{\beta_c} \Zz$.
\eproof

\subsection{End of proof}
\label{End}

We are now ready to prove our main result. First of all, let us fix one integer $c>1$ with the property that $\prod_{i \tei m, 1 \leq i<m} (1-\zeta^i)$ divides $c$. In addition, we can choose $c_1,c_2,\dotsc$ in $K \reg$ such that $c,c_1,c_2,\dotsc$ are free generators of a free abelian subgroup $\Gamma$ of $K \reg$ with $K \reg = \mu \times \Gamma$. Let $\Gamma_j \defeq \spkl{c,c_1,\dotsc,c_j}$ ($\Gamma_0 = \spkl{c}$).

\bprop
\label{betac=id}
The canonical homomorphism
\bgloz
  C_0(\Az_{\infty}) \rtimes (\mu \times \Gamma_j) \ri C_0(\Az_{\infty}) \rtimes K \rtimes (\mu \times \Gamma_j)
\egloz
is a rational isomorphism for all $i \geq 0$. Moreover, $\beta_{c_{j+1}}$, the automorphism induced by multiplication with $c_{j+1}$, is the identity on $K_*(C_0(\Az_{\infty}) \rtimes K \rtimes (\mu \times \Gamma_j))$.
\eprop
Here $K_*$ stands for the $\Zz / 2 \Zz$-graded abelian group $K_0 \oplus K_1$.
\bproof
We proceed inductively on $j$. Let us start with $j=0$. The Pimsner-Voiculescu exact sequence gives the following exact sequences:
\bglnoz
  \gekl{0} &\ri& K_1(C_0(\Az_{\infty}) \rtimes (\mu \times \Gamma_0)) \ri K_0(C_0(\Az_{\infty}) \rtimes \mu) \\
  &\xrightarrow{\id - (\beta_c)_* = 0}& K_0(C_0(\Az_{\infty}) \rtimes \mu) \ri K_0(C_0(\Az_{\infty}) \rtimes (\mu \times \Gamma_0)) \ri \gekl{0}
\eglnoz
and
\bgln
\label{2ndexseq}
	\gekl{0} &\ri& K_1(C_0(\Az_{\infty}) \rtimes K \rtimes (\mu \times \Gamma_0)) \\
	&\ri& K_0(C_0(\Az_{\infty}) \rtimes K \rtimes \mu) \nonumber \\
	&\xrightarrow{\id - (\beta_c)_* = 
  \rukl{
  \begin{smallmatrix}
  0&0 \\
  0&\id - 
  \rukl{
  \begin{smallmatrix}
  c^{-n} & * \\
  0 & \ddots
  \end{smallmatrix}
  }
  \end{smallmatrix}
  }}& K_0(C_0(\Az_{\infty}) \rtimes K \rtimes \mu) \nonumber \\
  &\ri& K_0(C_0(\Az_{\infty}) \rtimes K \rtimes (\mu \times \Gamma_0)) \ri \gekl{0} \nonumber
\egln
In the first sequence, we have $(\beta_c)_* = \id$ because $\beta_c$ is homotopic to the identity on $C_0(\Az_{\infty}) \rtimes \mu$. In the second sequence, we have plugged in the matrix describing $(\beta_c)_*$ given in \eqref{betac}.

Moreover, naturality of the Pimsner-Voiculescu sequence allows us to connect these two sequences by homomorphisms on K-theory induced by the canonical maps
\bgloz
  C_0(\Az_{\infty}) \rtimes \mu \ri C_0(\Az_{\infty}) \rtimes K \rtimes \mu
\egloz
and
\bgloz
  C_0(\Az_{\infty}) \rtimes (\mu \times \Gamma_0) \ri C_0(\Az_{\infty}) \rtimes K \rtimes (\mu \times \Gamma_0).
\egloz
By Proposition~\ref{K_0-inj}, we know that $C_0(\Az_{\infty}) \rtimes \mu \ri C_0(\Az_{\infty}) \rtimes K \rtimes \mu$ induces an injective map on $K_0$ whose image is contained in the copy of $\Zz^m$ in $K_0(C_0(\Az_{\infty}) \rtimes K \rtimes \mu)$. Comparing the ranks, we deduce that this injective map must be a rational isomorphism when we restrict its image to the copy of $\Zz^m$. Since this copy is isomorphic to $K_i(C_0(\Az_{\infty}) \rtimes K \rtimes (\mu \times \Gamma_0))$ (for $i=0,1$) via the maps in the second exact sequence~\eqref{2ndexseq}, we obtain that the canonical homomorphism $C_0(\Az_{\infty}) \rtimes (\mu \times \Gamma_0) \ri C_0(\Az_{\infty}) \rtimes K \rtimes (\mu \times \Gamma_0)$ induces a rational isomorphism on K-theory.

Now we know that $\beta_{c_1}$ is homotopic to the identity on $C_0(\Az_{\infty}) \rtimes (\mu \times \Gamma_0)$. This implies $(\beta_{c_1})_* = \id$ on
$K_*(C_0(\Az_{\infty}) \rtimes (\mu \times \Gamma_0)).$
As the canonical homomorphism
$C_0(\Az_{\infty}) \rtimes (\mu \times \Gamma_0) \ri C_0(\Az_{\infty}) \rtimes K \rtimes (\mu \times \Gamma_0)$
induces a rational isomorphism on K-theory (see above), we know that
\bgloz
  (\beta_{c_1})_* \otimes_{\Zz} \id_{\Qz} = \id \otimes_{\Zz} \id_{\Qz} \text{ on } K_*(C_0(\Az_{\infty}) \rtimes K \rtimes (\mu \times \Gamma_0)) \otimes_{\Zz} \Qz.
\egloz
But as
$K_*(C_0(\Az_{\infty}) \rtimes K \rtimes (\mu \times \Gamma_0))$
is free abelian (see Corollary~\ref{K(inf-with-c)}, $\Gamma_0 = \spkl{c}$), it follows that $(\beta_{c_1})_*$ must be the identity on
$K_*(C_0(\Az_{\infty}) \rtimes K \rtimes (\mu \times \Gamma_0))$.
This settles the case $j=0$.

The remaining induction step is proven in a similar way. We just have to use that $\beta_{c_{j+1}}$ is homotopic to the identity on $C_0(\Az_{\infty}) \rtimes (\mu \times \Gamma_j)$.
\eproof

\bcor
For every $i \in \Zz \nneg$, we can identify
\bgloz
  K_*(C_0(\Az_{\infty}) \rtimes K \rtimes (\mu \times \Gamma_j)) \text{ with }  K_0(C^*(\mu)) \otimes_{\Zz} \extalg(\Gamma_j)
\egloz
in such a way that the map on K-theory induced by the canonical homomorphism
\bgloz
  C_0(\Az_{\infty}) \rtimes K \rtimes (\mu \times \Gamma_j) \ri C_0(\Az_{\infty}) \rtimes K \rtimes (\mu \times \Gamma_{j+1})
\egloz
corresponds to the canonical map
\bgloz
  K_0(C^*(\mu)) \otimes_{\Zz} \extalg(\Gamma_j) \ri K_0(C^*(\mu)) \otimes_{\Zz} \extalg(\Gamma_{j+1})
\egloz
induced by the inclusion $\Gamma_j \into \Gamma_{j+1}$ for all $j \in \Zz \nneg$.
\ecor
\bproof
This follows inductively on $j$ using $(\beta_{c_{j+1}})_* = \id \text{ on } K_*(C_0(\Az_{\infty}) \rtimes K \rtimes (\mu \times \Gamma_j))$ (see Proposition~\ref{betac=id}) and the Pimsner-Voiculescu exact sequence. The induction starts with Corollary~\ref{K(inf-with-c)}.
\eproof
Finally, we arrive at
\bglnoz
  && K_*(C_0(\Az_{\infty}) \rtimes K \rtimes K \reg \cong \ilim_j K_*(C_0(\Az_{\infty}) \rtimes K \rtimes (\mu \times \Gamma_j) \\
  &\cong& \ilim_j K_0(C^*(\mu)) \otimes_{\Zz} \extalg(\Gamma_j) \cong K_0(C^*(\mu)) \otimes_{\Zz} \extalg(\Gamma).
\eglnoz
This completes the proof of our main result, Theorem~\ref{thm1}.

\bproof[Proof of Corollary~\ref{cor}]
By \cite[Theorem~3.6]{Cu-Li1}, the ring C*-algebras of rings of integers are simple and purely infinite, and by Corollaries~3 and 4 in \cite{Li}, these ring C*-algebras are nuclear and satisfy the UCT. Moreover, these ring C*-algebras are obviously unital and separable, and it is easy to see that for these algebras, the class of the unit in $K_0$ vanishes. Thus \cite[Theorem~8.4.1~(iv)]{Ror} tells us that two such ring C*-algebras are isomorphic if and only if their K-groups are isomorphic. Now our corollary follows from Theorem~\ref{thm2}.
\eproof

\bremark
In \cite[Remark~6.6]{Cu-Li2}, it was observed that the same ideas which lead to the duality theorem also yield
\bgl
\label{Mor_Afull}
  C_0(\Az) \rtimes K \rtimes K\reg \sim_M C^*(K \rtimes K\reg).
\egl
With the same strategy as in the proof of Theorem~\ref{thm1}, we can now compute $K_*(C^*(K \rtimes K\reg))$. We start with computing $K_*(C^*(K \rtimes \mu))$. To this end, we write $C^*(K \rtimes \mu)$ as an inductive limit where all the C*-algebras are given by $C^*(R \rtimes \mu)$ and the connecting maps are induced by multiplication with elements from $R\reg$. We can then use Theorem~\ref{LL} to determine the corresponding inductive limit in K-theory. To complete our computation, we proceed in an analogous manner as in Paragraph~\ref{End}. We choose a free abelian subgroup $\Gamma$ of $K\reg$ such that $K\reg = \mu \times \Gamma$ and then use the Pimsner-Voiculescu sequence iteratively. As a final result, we obtain that the canonical homomorphism $C^*(K\reg) \to C^*(K \rtimes K\reg)$ induces an isomorphism on K-theory. Thus, for every number field $K$, we obtain
\bgloz
  K_*(C_0(\Az) \rtimes K \rtimes K\reg) \cong K_*(C^*(K \rtimes K\reg)) \cong K_*(C^*(K\reg)) \cong K_0(C^*(\mu)) \otimes_{\Zz} \extalg(\Gamma).
\egloz
\eremark

\section{Injectivity of certain inclusions on K-theory}
\label{injectivity}
We want to prove
\bprop
\label{K_0-inj}
For every number field $K$, the homomorphism
\bgloz
  K_0(C_0(\Az_{\infty}) \rtimes \mu) \to K_0(C_0(\Az_{\infty}) \rtimes K \rtimes \mu)
\egloz
induced by the canonical map $C_0(\Az_{\infty}) \rtimes \mu \to C_0(\Az_{\infty}) \rtimes K \rtimes \mu$ is injective.
\eprop
Recall that $K \rtimes \mu$ acts on $C_0(\Az_{\infty})$ via affine transformations as in \S~\ref{review}. This proposition is needed in the proof of Theorem~\ref{thm1}; more precisely, it is needed in the proofs of Corollary~\ref{K-subalg'} and Proposition~\ref{betac=id}. We have postponed the proof of this proposition until now because it is independent from the previous sections.

\subsection{Induction and restriction }
\label{sec:Induction_and_restriction}

In this section let $G$ be a discrete group and let $A$ be a $G$-C*-algebra, i.e., a C*-algebra $A$ with left $G$-action. We write $g \bullet a$ for this action. We write elements in $C_c(G,A)$ as finite sums of the form $\sum_g a_g \cdot g$. Let us present some elementary facts about induction and restriction which hold for reduced as well as full crossed products. But as we will consider amenable groups anyway later on, we only treat the case of reduced crossed products and remark that full crossed products can be studied in a similar way.

Let $\iota: H \into G$ be an injective group homomorphism.  The homomorphism $\id_A \rtimes_r \, \iota: A \rtimes_r H \to A \rtimes_r G$ induces the map called \emph{induction with $\iota$}, 
\bgl
  \iota_* = \ind_\iota: K_i(A \rtimes_r H) \to K_i(A \rtimes_r G).
  \label{induction_with_iota_twisted_r}
\egl

Now suppose that the index of the image of $\iota$ in $G$ is finite. We want to construct maps in the \an{wrong} direction, i.e., a map $K_i(A \rtimes_r G) \to K_i(A \rtimes_r H)$. To simplify notations, we think of $H$ as a subgroup of $G$ via $\iota$. On $K_0$, we proceed as follows:

We obtain an isomorphism of (left) $A \rtimes_r H$-modules
\bgl
  \bigoplus_{\gamma H \in G/H} A \rtimes_r H \xrightarrow{\cong} \res_{A \rtimes_r G}^{A \rtimes_r H} (A \rtimes_r G)
  \label{restriction_A-rtimes_G_to_A-rtimes_H_r}
\egl
sending $(x_{\gamma H})_{\gamma H}$ to $\sum_{\gamma H \in G/H} x_{\gamma H} \cdot \gamma^{-1}$ after choices of representatives $\gamma \in \gamma H$ for every $\gamma H \in G/H$. Hence $A \rtimes_r G$ is a finitely generated free $A \rtimes_r H$-module. This implies that the restriction of every finite generated projective $A \rtimes_r G$-module to $A \rtimes_r H$ is again a finitely generated projective $A \rtimes_r H$-module. Hence we obtain a homomorphism $\iota^* = \res_\iota: K_0(A \rtimes_r G) \to K_0(A \rtimes_r H)$ which is called \emph{restriction with $\iota$}.

Here is an alternative construction which has the advantage that it works for $K_1$ as well: First of all, we represent $A$ faithfully on a Hilbert space $\cH$. Then $A \rtimes_r G$ is faithfully represented on $\cH \otimes \ell^2(G)$ via $(a \cdot g)(\xi \otimes \varepsilon_\gamma) = (((g \gamma)^{-1} \bullet a) \xi) \otimes \varepsilon_{g \gamma}$. We identify $A \rtimes_r G$ with concrete operators on $\cH \otimes \ell^2(G)$ via this representation. Now fix representatives $\gamma \in \gamma H$ for every $\gamma H \in G/H$. From the (set-theoretical) bijection $G = \dotcup_{\gamma} \gamma H \cong \dotcup_{\gamma} H$ we obtain a unitary
\bglnoz
  \cH \otimes \ell^2(G) &\cong& \bigoplus_{G/H} \cH \otimes \ell^2(H) \\
  \sum_{\gamma} \sum_h \lambda_{\gamma h} \xi_{\gamma h} \otimes \varepsilon_{\gamma h} 
  &\ma& \rukl{\sum_h \lambda_{\gamma h} \xi_{\gamma h} \otimes \varepsilon_h}_\gamma.
\eglnoz
Conjugation by this unitary yields the identification
\bgloz
  \cL(\cH \otimes \ell^2(G)) \cong M_{[G:H]}(\cL(\cH \otimes \ell^2(H))), \ T \ma \rukl{P_{\ell^2(H)} \gamma^{-1} T \gamma' P_{\ell^2(H)}}_{\gamma,\gamma'}
\egloz
where $P_{\ell^2(H)}$ is the orthogonal projection onto the subspace $\ell^2(H)$ of $\ell^2(G)$.

A straightforward computation shows that this isomorphism sends the operator $a \cdot g$ to the matrix whose $(\gamma,\gamma')$-th entry is $(\gamma^{-1} \bullet a) \cdot (\gamma^{-1} g \gamma')$ if $\gamma^{-1} g \gamma'$ lies in $H$ and $0$ if $\gamma^{-1} g \gamma' \notin H$. In particular, $A \rtimes_r G$ is mapped to $M_{[G:H]}(A \rtimes_r H)$. This homomorphism induces the desired map $\iota^*: K_i(A \rtimes_r G) \to K_i(M_{[G:H]}(A \rtimes_r H)) \cong K_i(A \rtimes_r H)$.

Moreover, these restriction maps do not depend on the choices of the representatives $\gamma$ in $\gamma H \in G/H$. The reason is that for two different choices, the constructed homomorphisms $A \rtimes_r G \to M_{[G:H]}(A \rtimes_r H)$ turn out to be unitarily equivalent, hence they induce the same map in K-theory.

If $\iota: H \to G$ is an inclusion of subgroups, one often writes $\res_\iota = \res_G^H$ and $\ind_\iota = \ind_H^G$.  For $g \in G$ conjugation defines an isomorphism of C*-algebras $c(g) :  A \rtimes G \to A \rtimes G$ sending $x \in A \rtimes G$ to $gxg^{-1}$. The two endomorphisms $\ind_{c(g)}$ and $\res_{c(g)}$ of $K_0(A \rtimes G)$ are the identity. Hence in the next lemma the choice of representatives $\gamma \in H \gamma K$ for an element $H \gamma K \in H \backslash G/K$ does not matter. It is a variation of the classical Double Coset Formula.

\begin{lemma} \label{lem:Double_Coset_Formula} Let $H, K \subseteq A$ be two subgroups of $G$. Suppose that $H$ has finite index in $G$. Then, for $i=0,1$, we get the following equality of homomorphisms $K_i(A \rtimes_r K) \to K_i(A \rtimes_r G)$:
$$\res_G^H \circ \ind_K^G  = \sum_{H \gamma K \in H \backslash G/K} \ind_{c(\gamma): K \cap \gamma^{-1} H \gamma \to H} \circ \res_K^{K \cap \gamma^{-1} H \gamma},$$
where $c(\gamma)$ is conjugation with $\gamma$, i.e., $c(\gamma)(k) = \gamma k \gamma^{-1}$.
\end{lemma}
\begin{proof}
Since $H$ has finite index in $G$, $K \cap \gamma^{-1} H \gamma$ has finite index in $K$ and $H \backslash G/K$ is finite. Hence the expression appearing in Lemma~\ref{lem:Double_Coset_Formula} makes sense.

On $K_0$, we can proceed as follows: Let $P$ be a finitely generated projective $A \rtimes_r K$-module. Fix choices of representatives $\gamma \in H \gamma K$ for every $H \gamma K \in H \backslash G/K$. Next we claim that the following homomorphism
  \begin{multline*}
    \bigoplus_{H \gamma K \in H \backslash G/K} \ind_{c(\gamma)}
    \rukl{\res_{K}^{K \cap \gamma^{-1} H \gamma} P} 
    = \bigoplus_{H \gamma K \in H\backslash G/K} A \rtimes_r H \otimes_{A \rtimes_r (K \cap \gamma^{-1} H \gamma)} P
    \\
    \xrightarrow{\cong} \res_G^H \circ \ind_K^G P = A \rtimes_r G \otimes_{A \rtimes_r K} P.
  \end{multline*}
is an isomorphism of $A \rtimes_r K$-modules. Its restriction to the summand for $H \gamma K \in H \backslash G/K$ sends $x \otimes p$ for $x \in A \rtimes_r K$ and $p \in P$ to $x \gamma \otimes p$. Since it is natural and compatible with direct sums, it suffices to show bijectivity for $P = A \rtimes_r H$ what is straightforward.

Again, we present an alternative proof which works in general (i.e., for $i=1$ as well). Choose representatives $\gamma \in H \gamma K$ for every $H \gamma K \in H \backslash G/K$. For every such $\gamma$, choose representatives $\kappa_\gamma \in \kappa_\gamma (K \cap \gamma^{-1} H \gamma)$ for every $\kappa_\gamma (K \cap \gamma^{-1} H \gamma) \in K/(K \cap \gamma^{-1} H \gamma)$. The first observation is that the products $\kappa_\gamma \gamma^{-1}$ form a full set of representatives for $G/H$, i.e., we can write $G$ as a disjoint union as follows:
\bgloz
  G = \bigcup_\gamma \bigcup_{\kappa_\gamma} (\kappa_\gamma \gamma^{-1}) H.
\egloz
Now we use the representatives $\gekl{\kappa_\gamma \gamma^{-1}}$ of $G/H$ to construct as above the homomorphism $A \rtimes_r G \to M_{[G:H]}(A \rtimes_r H)$ which induces $\res_H^G$ on K-theory. The composition of this map with the canonical map $A \rtimes_r K \to A \rtimes_r G$ is given by
\bgloz
  C_c(K,A) \ni a \cdot k \ma (x_{\kappa_\gamma \gamma^{-1},\kappa'_{\gamma'} \gamma'^{-1}}) \in M_{[G:H]}(A \rtimes_r H)
\egloz
with
\bgloz
  x_{\kappa_\gamma \gamma^{-1},\kappa_{\gamma'} \gamma'^{-1}} = 
  \bfa
  ((\gamma \kappa_\gamma^{-1}) \bullet a) \cdot (\gamma \kappa_\gamma^{-1} k \kappa'_{\gamma'} \gamma'^{-1}) 
  \falls \gamma \kappa_\gamma^{-1} k \kappa_{\gamma'} \gamma'^{-1} \in H, \\
  0 \sonst.
  \efa
\egloz
The second observation is that the matrix $(x_{\kappa_\gamma \gamma^{-1},\kappa'_{\gamma'} \gamma'^{-1}})$ can be decomposed into smaller matrix blocks since for $\gamma \neq \gamma'$, $\gamma \kappa_\gamma^{-1} k \kappa_{\gamma'} \gamma'^{-1}$ does not lie in $H$ no matter which $k \in K$ we take. This holds because $\gamma \neq \gamma'$ implies $(H \gamma K) \cap (H \gamma' K) = \emptyset$ by our choice of the $\gamma$s. Hence in K-theory, we obtain that the class of $(x_{\kappa_\gamma \gamma^{-1},\kappa'_{\gamma'} \gamma'^{-1}})$ is the sum over $\gamma$ of the classes of $(x_{\kappa_\gamma \gamma^{-1},\kappa'_\gamma \gamma^{-1}})_{\kappa_\gamma,\kappa'_\gamma}$. 

Now the third observation is that
\bglnoz
  x_{\kappa_\gamma \gamma^{-1},\kappa'_\gamma \gamma^{-1}} 
  &=& 
  \bfa
  ((\gamma \kappa_\gamma^{-1}) \bullet a) \cdot (\gamma \kappa_\gamma^{-1} k \kappa'_\gamma \gamma^{-1}) 
  \falls \gamma \kappa_\gamma^{-1} k \kappa_\gamma \gamma^{-1} \in H \\
  0 \sonst
  \efa \\
  &=& 
  \bfa
  c(\gamma)((\kappa_\gamma^{-1} \bullet a) \cdot (\kappa_\gamma^{-1} k \kappa'_\gamma))
  \falls \kappa_\gamma^{-1} k \kappa_\gamma \in K \cap \gamma^{-1} H \gamma, \\
  0 \sonst.
  \efa
\eglnoz
This means that the map $a \cdot k \ma  (x_{\kappa_\gamma \gamma^{-1},\kappa_\gamma \gamma^{-1}})_{\kappa_\gamma,\kappa'_\gamma}$ is precisely the composition with $c(\gamma)$ (or rather the extension of $c(\gamma)$ to matrices) of one of the maps $A \rtimes_r K \to M_{[K:K \cap \gamma^{-1} H \gamma]}(A \rtimes_r (K \cap \gamma^{-1} H \gamma))$ which induce $\res_K^{K \cap \gamma^{-1} H \gamma}$. This proves the Double Coset Formula.
\end{proof}

Let $\iota: H \to G$ be the inclusion of a normal subgroup of finite index. Denote by $N_{G/H} \in \Zz[G/H]$ the norm element, i.e., $N_{G/H} = \sum_{\gamma H} \gamma H$.  If $M$ is any $\Zz[G/H]$-module, then multiplication with $N_{G/H}$ induces a map $\Zz \otimes_{\Zz[G/H]} M \to M^{G/H}$ whose kernel and whose cokernel are annihilated by multiplication with $[G:H]$. Denote by $c(g): A \rtimes_r H \to A \rtimes_r H$ and by $c(g): A \rtimes_r G \to A \rtimes_r G$ the ring homomorphisms obtained by conjugation with $g$, i.e., they send $x$ to $gxg^{-1}$. The induction homomorphism $\ind_{c(g)}: K_i(A \rtimes_r G) \to K_i(A \rtimes_r G)$ is the identity. The induction homomorphism $\ind_{c(g)}: K_i(A \rtimes_r H) \to K_i(A \rtimes_r H)$ is the identity provided that $g \in H$. Since $c(g_1) \circ c(g_2) = c(g_1g_2)$ holds for $g_1,g_2 \in G$ and $c(1) = \id$, we obtain a $G/H$-action on $K_i(A \rtimes_r H)$. The group homomorphisms $c(g) \circ i$ and $i \circ c(g)$ agree. Hence the map $\iota_* = \ind_\iota: K_i(A \rtimes_r H) \to K_i(A \rtimes_r G)$ factors over the canonical projection $K_i(A \rtimes_r H) \onto \Zz \otimes_{\Zz[G/H]} K_i(A \rtimes_r H)$ to a homomorphism $\overline{\iota_*}: \Zz \otimes_{\Zz[G/H]} K_i(A \rtimes_r H) \to K_i(A \rtimes_r G)$. In addition, the homomorphism $A \rtimes_r G \to M_{[G:H]}(A \rtimes_r H)$ which induces $\res_{\iota}$ commutes with $c(g)$ (the extended map on matrices) up to unitary equivalence, and therefore the map $\iota^* = \res_\iota: K_i(A \rtimes_r G) \to K_i(A \rtimes_r H)$ factors over the inclusion $K_i(A \rtimes_r H)^{G/H} \into K_i(A \rtimes_r H)$ to a map $\overline{\iota^*}: K_i(A \rtimes_r G) \to K_i(A \rtimes_r H)^{G/H}$. These considerations and Lemma~\ref{lem:Double_Coset_Formula} imply

\begin{lemma}\label{lem_double_coset_for_H_is_K_subset_G-normal_fin_index}
  Let $\iota: H \to G$ be the inclusion of a normal subgroup of finite index.  Then we obtain a commutative diagram such that the kernel and cokernel of the lower horizontal map are annihilated by $[G:H]$.
$$\xymatrix@C=1mm{
  K_i(A \rtimes_r H) \ar[rr]^{\iota_*}  \ar[dr] 
  & & K_i(A \rtimes_r G) \ar[rr]^{\iota^*} \ar[dr]^{\overline{\iota^*}} & &
  K_i(A \rtimes_r H)
  \\
  & \Zz \underset{\Zz[G/H]}{\otimes} K_i(A \rtimes_r H) \ar[rr]^{N_{G/H}}
  \ar[ur]^{\overline{\iota_*}} & & K_i(A \rtimes_r H)^{G/H} \ar[ur] & }
$$
\end{lemma}

Moreover, we have:
\bcor \label{cor:A_rtimes_finite} Let $F$ be a finite group and $A$ be a $F$-C*-algebra. Then the inclusion induces a map $K_i(A) \to K_i(A \rtimes F)$ that factors over the canonical projection $K_i(A) \to \Zz \otimes_{\Zz F} K_i(A)$ to a map $\Zz \otimes_{\Zz F} K_i(A) \to K_i(A \rtimes F)$ whose kernel is annihilated by multiplication with $\abs{F}$.
\ecor
\begin{proof}
This follows directly from Lemma~\ref{lem_double_coset_for_H_is_K_subset_G-normal_fin_index} applied to $H = \gekl{1}$ and $G=F$. Of course, since $F$ is finite, we do not need to distinguish between reduced and full crossed products.
\end{proof}

\subsection{Injectivity after inverting orders}

Now we consider the case $G = R \rtimes \mu$. Actually, we can treat a slightly more general situation, namely that $G = \Zz^n \rtimes F$ for a finite cyclic group $F$ where the conjugation action of $F$ on $\Zz^n$ is free outside the origin $0 \in \Zz^n$. Let $A$ be a $G$-C*-algebra. Since $G$ is amenable, we do not have to distinguish between full and reduced crossed products. It is clear that $A \rtimes G \cong (A \rtimes \Zz^n) \rtimes F$ where $\Zz^n$ acts on $A$ via the restricted action and $F$ acts on $A \rtimes \Zz^n$ via $c(f)$, i.e., $f \bullet (a \cdot g) = (f \bullet a) \cdot (fgf^{-1})$ for $a \in A$ and $g \in \Zz^n$.

\begin{theorem} \label{the:injectivity_for_F_to_Zn_rtimes_F}
Suppose that the map $K_i(A) \to K_i(A\rtimes \Zz^n)$ is injective after inverting $\abs{F}$. Then the map $K_i(A \rtimes F) \to K_i(A \rtimes G)$ induced by the canonical homomorphism $A\rtimes F \to A \rtimes G$ is injective after inverting $\abs{F}$.
\end{theorem}

The proof of this theorem needs some preparation.

\begin{lemma} \label{lem:A_rtimesG_for_Zn_to_G_to_F} 
Let $\cM$ be a complete system of representatives of conjugacy class of maximal finite subgroups of $G$.  Let $m$ be the least common multiple of the orders of the subgroups $M \in \cM$.

Then for every $M \in \cM$ the canonical homomorphisms $A \to A \rtimes M \to A \rtimes G$ induce a map 
\bgl
\label{bij?}
  \ker\bigl(\Zz \otimes_{\Zz M} K_i(A) \to K_i(A\rtimes G)\bigr) \to \ker\bigl(K_i(A \rtimes M) \to K_i(A\rtimes G)\bigr)
\egl
which is bijective after inverting $m$.
\end{lemma}
\begin{proof}
We have the following $G$-pushout (compare \cite[(4.5)]{La-Lü})
\begin{eqnarray*}
\xymatrix@C=15mm{
\coprod_{M \in \cM} G \times_{M}  EM 
\ar[r] 
\ar[d]_{\coprod_{M \in \cM} \id_G \times_{M} f_M}
& EG \ar[d]^f
\\
\coprod_{M \in \cM} G \times_M \pt
\ar[r]
& \eub{G}
}
\end{eqnarray*}
Applying equivariant K-homology with coefficients in $A$, we obtain a long exact sequence (the Mayer-Vietoris sequence associated with the pushout)
\begin{multline}
\cdots \to \bigoplus_{M \in \cM} K_i^G(G \times_M EM;A) \to K_i^G(EG;A) \oplus \bigoplus_{M \in \cM} K_i^G(G \times_M \pt;A) \\
\to K_i^G(\eub{G};A) \to \bigoplus_{M \in \cM} K_{i-1}^G(G \times_M EM;A) \\
\to K_{i-1}^G(EG;A) \oplus \bigoplus_{M \in \cM} K_{i-1}^G(G \times_M \pt;A) \to K_{i-1}^G(\eub{G};A) \to \cdots 
\label{long_exact_sequence_for_KG(Artimes_G)}
\end{multline}
There is a spectral sequence converging to $K_i^G(G \times_M EM;A)$
whose $E^2$-term is $H_p^G(G \times_M EM;K_q(A))$. Since $M$ is
finite, we know that $H_p^G(G \times_M EM;K_q(A)) \cong H_p^M(EM;K_q(A))$ is
annihilated by multiplication with $\abs{M}$ for $p \ge 1$.  Since 
$H_0^G(G \times_M EM;K_i(A))$ can be identified with $\Zz \otimes_{\Zz M} K_i(A)$, the
edge homomorphism
$$\Zz  \otimes_{\Zz M} K_i(A) \to  K_i^G(G \times_M EM;A)$$
is bijective after inverting $\abs{M}$. Its composite with
$$K_i^G(\id_G \times_M f_M): K_i^G(G \times_M EM;A) 
\to K_i^G(G \times_M \pt;A) \cong K_i(A \rtimes M)$$ 
is the map
$\Zz  \otimes_{\Zz M} K_i(A) \to K_i(A \rtimes M)$
induced by the canonical homomorphism $A \to A \rtimes M$. 
The kernel of this map 
$\Zz  \otimes_{\Zz M} K_i(A) \to K_i(A \rtimes M)$
is annihilated by multiplication with $\abs{M}$ by 
Corollary~\ref{cor:A_rtimes_finite}. This already implies injectivity of the map in \eqref{bij?} after inverting $m$. But this also implies that the long exact 
sequence~\eqref{long_exact_sequence_for_KG(Artimes_G)}
yields after inverting $m$ the short exact sequence
\begin{multline}
\label{ses}
0 \to \bigoplus_{M} \Zz \otimes_{\Zz M} K_i(A)[\tfrac{1}{m}] 
\\
\to 
K_i^G(EG;A)[\tfrac{1}{m}] \oplus \bigoplus_{M \in \cM} K_i(A \rtimes M)[\tfrac{1}{m}]
\\
\to
K_i(A \rtimes G)[\tfrac{1}{m}] \to 0
\end{multline}
where we used that $G$ is amenable, hence satisfies the Baum-Connes conjecture with coefficients, i.e., $\asmb_i: K_i^G(\eub{G};A) \to K_i(A \rtimes G)$ is an isomorphism (see \cite{Hig-Kas}). Exactness of \eqref{ses} immediately yields surjectivity of the map in \eqref{bij?} because $\bigoplus_{M} \Zz \otimes_{\Zz M} K_i(A)[\tfrac{1}{m}] \to \bigoplus_{M \in \cM} K_i(A \rtimes M)[\tfrac{1}{m}]$ is induced by the canonical homomorphisms $A \to A \rtimes M$ as explained above and $\bigoplus_{M \in \cM} K_i(A \rtimes M)[\tfrac{1}{m}] \to K_i(A \rtimes G)[\tfrac{1}{m}]$ is induced by the canonical homomorphisms $A \rtimes M \to A \rtimes G$. This proves our claim.
\end{proof}

\begin{lemma} \label{lem:injectivity_for_F_to_Zn_rtimes_F_otimes}
For $i=0$ or $1$, suppose that the map $K_i(A) \to K_i(A\rtimes \Zz^n)$ is injective after inverting $\abs{F}$. Then the map $\Zz \otimes_{\Zz F} K_i(A) \to \Zz \otimes_{\Zz F} K_i(A \rtimes \Zz^n)$ coming from the canonical homomorphism $A \to A \rtimes \Zz^n$ is injective after inverting $\abs{F}$.
\end{lemma}
\begin{proof}
Consider the following commutative diagram
$$\xymatrix{
\Zz \otimes_{\Zz F} K_i(A) \ar[r] \ar[d]^{N_F}
&
\Zz \otimes_{\Zz F} K_i(A \rtimes \Zz^n)\ar[d]^{N_F}
\\
K_i(A)^F \ar[r] \ar[d] 
& 
K_i(A \rtimes \Zz^n)^F \ar[d]
\\
K_i(A) \ar[r] 
& 
K_i(A \rtimes \Zz^n)}
$$
The vertical maps denoted by $N_F$ are given by multiplication with the norm element $N_F$ and are isomorphisms after inverting $\abs{F}$. The two lower vertical arrows are the canonical inclusions. The horizontal arrows are induced by the canonical homomorphism $A \to A \rtimes \Zz^n$. Since the lower horizontal arrow is injective after inverting $\abs{F}$ by assumption, the same is true for the upper horizontal arrow.
\end{proof}

\begin{proof}[Proof of Theorem~\ref{the:injectivity_for_F_to_Zn_rtimes_F}]
By assumption, the canonical homomorphism $A \to A \rtimes \Zz^n$ is injective in K-theory once we invert $\abs{F}$. Hence Lemma~~\ref{lem:injectivity_for_F_to_Zn_rtimes_F_otimes} implies that $\Zz \otimes_{\Zz F} K_i(A) \to \Zz \otimes_{\Zz F} K_i(A \rtimes \Zz^n)$ is injective after inverting $\abs{F}$. Now Corollary~\ref{cor:A_rtimes_finite} tells us that the map
$\Zz \otimes_{\Zz F} K_i(A \rtimes \Zz^n) \to K_i((A \rtimes \Zz^n) \rtimes F) \cong K_i(A \rtimes G)$ 
is injective after inverting $\abs{F}$. 
Hence $\Zz \otimes_{\Zz F} K_i(A) \to K_i(A \rtimes G)$ is injective after inverting $\abs{F}$. 
Finally apply Lemma~\ref{lem:A_rtimesG_for_Zn_to_G_to_F} in the case $M = F$.
\end{proof}

\subsection{Injectivity}
Finally, we are ready for the
\bproof[Proof of Proposition~\ref{K_0-inj}]
Let $c$ be some element in $R \reg$. Since the additive action of $c^{-1} R$ is homotopic to the trivial action, an iterative application of the Pimsner-Voiculesu sequence implies that the canonical homomorphism $C_0(\Az_{\infty}) \to C_0(\Az_{\infty}) \rtimes (c^{-1} R)$ is injective on $K_0$. Thus Theorem~\ref{the:injectivity_for_F_to_Zn_rtimes_F} yields that $C_0(\Az_{\infty}) \rtimes \mu \to C_0(\Az_{\infty}) \rtimes (c^{-1} R) \rtimes \mu$ is injective on $K_0$ after inverting $\abs{\mu}$. By equivariant Bott periodicity (see \cite[Theorem~20.3.2]{Bla}, $n$ is even), we know that $K_0(C_0(\Az_{\infty}) \rtimes \mu) \cong K_0(C^*(\mu))$ and the latter group is free abelian. Thus the canonical homomorphism $C_0(\Az_{\infty}) \rtimes \mu \to C_0(\Az_{\infty}) \rtimes (c^{-1} R) \rtimes \mu$ itself must be injective on $K_0$. But then, since
\bgloz
  C_0(\Az_{\infty}) \rtimes K \rtimes \mu = \overline{\bigcup_{c \in R \reg} C_0(\Az_{\infty}) \rtimes (c^{-1} R) \rtimes \mu},
\egloz
the canonical homomorphism $C_0(\Az_{\infty}) \rtimes \mu \to C_0(\Az_{\infty}) \rtimes K \rtimes \mu$ must be injective on $K_0$ as well by continuity of $K_0$.
\eproof

\end{document}